\documentclass[leqno,10pt]{amsart}

\usepackage{amsmath,amsfonts,amsthm,amssymb}
\usepackage{verbatim}
\usepackage{relsize}
\usepackage{color}
\usepackage
{enumitem}
\usepackage{color}
\usepackage{geometry}
\theoremstyle{plain}
\newtheorem{teor}{Theorem}
\newtheorem{prop}[teor]{Proposition}
\newtheorem{lema}[teor]{Lemma}
\newtheorem{coro}[teor]{Corollary}

\newtheorem{teorintro}{Theorem}

\newtheorem{corointro}[teorintro]{Corollary}

\theoremstyle{definition}
\newtheorem{defi}[teor]{Definition}
\newtheorem{rem}[teor]{Remark}
\newtheorem*{rem*}{Remark}

\newcommand{\tcr}{\textcolor{black}}
\newcommand{\tcred}{\textcolor{black}}

\newcommand{\KN}{\mathbin{\bigcirc\mspace{-15mu}\wedge\mspace{3mu}}}

\newenvironment{demo}{\rm \trivlist \item[\hskip \labelsep{\it
      Proof}.]}{\nopagebreak \hfill $\square$ \endtrivlist}

\newcommand{\Ric}{\mbox{{\rm Ric}}}

\title{Stable maximal hypersurfaces in Lorentzian spacetimes}

\author[G. Colombo]{Giulio Colombo}
\address{Dipartimento di Matematica, Universit\`a degli Studi di Milano, 20133 Milano, Italy}
\curraddr{}
\email{giulio.colombo@unimi.it}
\thanks{}

\author[J. A. S. Pelegr\'in]{Jos\'e A. S. Pelegr\'in}
\address{Departamento de Geometr\'ia y Topolog\'ia, Universidad de Granada, 18071 Granada, Spain}
\curraddr{Departamento de Matem\'atica Aplicada y Estad\'istica, Universidad 
	CEU San Pablo, 28003 Madrid, Spain}
\email{jpelegrin@ugr.es, jose.sanchezpelegrin@ceu.es}
\thanks{}

\author[M. Rigoli]{Marco Rigoli}
\address{Dipartimento di Matematica, Universit\`a degli Studi di Milano, 20133 Milano, Italy}
\curraddr{}
\email{marco.rigoli@unimi.it}
\thanks{}

\date{}

\begin{document}

\maketitle

\thispagestyle{empty}

\begin{abstract}
We study the geometry of stable maximal hypersurfaces in a variety of spacetimes satisfying various physically relevant curvature assumptions, for instance the Timelike Convergence Condition (TCC). We characterize stability when the target space has constant sectional curvature as well as give sufficient conditions on the geometry of the ambient spacetime (e.g., the validity of TCC) to ensure stability. Some rigidity results and height estimates are also proven in GRW spacetimes. In the last part of the paper we consider $k$-stability of spacelike hypersurfaces, a concept related to mean curvatures of higher orders.
\end{abstract}
\vspace*{5mm}

\noindent \textbf{MSC 2010} \; {
	Primary: 53C24, 
	53C42, 
	35J20; 
	Secondary: 35P15, 
	53C50, 
	53C80, 
	83C99 
}

\noindent \textbf{Keywords} \; {
	Stable maximal hypersurface $\cdot$
	$k$-stability $\cdot$
	Spacetime of constant sectional curvature $\cdot$
	Generalized Robertson-Walker spacetime
}

\section{Introduction}

In the last decades, maximal hypersurfaces in spacetimes have attracted a great deal of mathematical and physical interest. The importance of this family of spacelike hypersurfaces in General Relativity is well-known and a summary of several reasons justifying this opinion can be found, for instance, in \cite{M-T}. Among them, we emphasize the key role they play in the study of the Cauchy problem \cite{CG, Li} as well as their importance in the proof of the positivity of the gravitational mass \cite{SY}. Furthermore, maximal hypersurfaces describe, in some relevant cases, the transition between the expanding and contracting phases of a relativistic universe. Finally, the existence of constant mean curvature (and in particular maximal) hypersurfaces is useful in the study of the structure of singularities in the space of solutions of the Einstein equations \cite{AMM}. At last, we should also mention their use in numerical relativity for integrating forward in time \cite{JKG}.

From a mathematical point of view, maximal hypersurfaces in a spacetime $\overline{M}$ are (locally) critical points for a natural variational problem, namely, that of the area functional (see, for instance, \cite{brasil}) and their study is helpful for understanding the structure of $\overline{M}$ \cite{Bar}. In particular, for some asymptotically flat spacetimes, maximal hypersurfaces produce a foliation of the spacetime, defining a time function \cite{BF}. Classical papers dealing with uniqueness of maximal hypersurfaces are, for instance, \cite{BF, Ch}, although a previous relevant result in this direction was the proof given by Cheng and Yau \cite{CY} of the Bernstein-Calabi conjecture \cite{Ca}: spacelike affine hyperplanes are the only complete maximal hypersurfaces in the the $(m+1)$-dimensional Lorentz-Minkowski spacetime. Nishikawa \cite{N} extended their result by proving that any complete maximal hypersurface immersed in a spacetime $\overline{M}$ is totally geodesic when $\overline{M}$ belongs to a family of locally symmetric Lorentzian manifolds that includes spacetimes of nonnegative constant curvature. Ishihara \cite{I} showed that this property is not shared by spacetimes of negative constant curvature by exhibiting an example of a complete maximal hypersurface with constant nonzero norm of the shape operator in the $(m+1)$-dimensional anti-de Sitter spacetime of curvature $-1$. In Theorem \ref{teoni2} below we prove a slight generalization of Nishikawa's result by proving an upper bound on the norm of the shape operator first obtained by Ishihara in the case of ambient spacetimes of constant curvature. More recently, new uniqueness results for maximal hypersurfaces have been found in a large variety of spacetimes by means of different techniques \cite{ARS, PRR1, PRR2}.

In this paper we will focus on a particular family of maximal hypersurfaces, namely, stable maximal hypersurfaces, that is, critical points of the volume functional for compactly supported normal variations with non-positive second variation. A mild condition on the curvature of the ambient spacetime is enough to ensure stability of maximal hypersurfaces.

\begin{teorintro}
	\label{intro:teostab}
	Let $\overline{M}$ be a spacetime with nonnegative Ricci curvature on timelike vectors. If $\psi: M \longrightarrow \overline{M}$ is a (not necessarily complete) \tcred{oriented} maximal hypersurface, then $\psi$ is stable. If $M$ is also compact, then $\psi$ is totally geodesic.
\end{teorintro}

\noindent Note that in an oriented spacetime $\overline{M}$, the time orientation of $\overline{M}$ ensures that every spacelike hypersurface is oriented. In General Relativity, a spacetime with nonnegative Ricci curvature on timelike vectors 
is said to obey the Timelike Convergent Condition (TCC). It is usually argued that the TCC is the mathematical way to express that gravity, on average, attracts (see \cite{O'N}). Theorem \ref{intro:teostab} generalizes Corollary 5.6 of \cite{ARS} and Theorem 1 of \cite{PRR1}, where the authors show that compact maximal hypersurfaces in a spacetime $\overline{M}$ obeying the TCC are totally geodesic by also assuming the existence of certain infinitesimal symmetries in $\overline{M}$. As a corollary, we also have an alternative proof of Theorem 4.1 of \cite{BF}, a uniqueness result for vacuum spacetimes.

\begin{corointro}
	\label{intro:corovac}
	Let $M$ be a compact maximal hypersurface in a spacetime that obeys the Einstein vacuum equations without cosmological constant. Then, $M$ is totally geodesic.
\end{corointro}

\noindent If a maximal hypersurface $\psi:M\to\overline{M}$ is unstable, then there exist spacelike hypersurfaces of larger volume in $\overline{M}$ nearby $\psi$. This happens, for instance, for the equator of de Sitter spacetime, which is a saddle point for the volume functional. In fact, we have the following

\begin{teorintro}
	\label{intro:teocsc}
	Let $\overline{M}$ be an $(m+1)$-dimensional spacetime of constant curvature $\overline{\kappa}$ and let $\psi : M \longrightarrow \overline{M}$ be a complete \tcred{oriented} maximal hypersurface.
	\begin{itemize}[noitemsep,nolistsep]
		\item [i)] If $\overline{\kappa}>0$ then $M$ is compact and the immersion $\psi$ is totally geodesic and unstable.
		\item [ii)] If $\overline{\kappa}=0$ then $\psi$ is totally geodesic and stable.
		\item [iii)] If $\overline{\kappa}<0$ then $\psi$ is stable and the shape operator $A$ and the scalar curvature $\mathrm{S}$ of $M$ satisfy
		\[
			\mathrm{trace}(A^2)\leq -m\overline{\kappa}, \qquad \mathrm{S} \leq (m-2)m\overline{\kappa}.
		\]
		If $M$ is also compact, then $\psi$ is totally geodesic.
	\end{itemize}
\end{teorintro}

\noindent In case the ambient manifold $\overline{M}$ is a $3$-dimensional spacetime of constant curvature, we are able to provide more information on the topology of \tcred{orientable} complete maximal surfaces.

\begin{teorintro}
	\label{intro:teosup1}
	Let $\overline{M}$ be a $3$-dimensional spacetime of constant curvature $\overline{\kappa}$, $M$ a complete maximal \tcred{oriented} surface in $\overline{M}$.
	\begin{itemize}[noitemsep,nolistsep]
		\item [i)] If $\overline{\kappa}>0$ then $M$ is a totally geodesic, unstable round sphere of constant curvature $\overline{\kappa}$.
		\item [ii)] If $\overline{\kappa}=0$ then $M$ is totally geodesic, stable and it is either a Euclidean plane, or a flat cylinder or a flat torus.
		\item [iii)] If $\overline{\kappa}<0$ then $M$ is stable and has non-positive Gaussian curvature. If $M$ is compact then it is totally geodesic and its Euler characteristic satisfies
		\[
			\chi(M) = \frac{\overline{\kappa}}{2\pi}\,\mathrm{Vol}(M).
		\]
		If $M$ is non-compact but its total curvature and its Euler characteristic are finite, then
		\[
			\chi(M) \geq \frac{\overline{\kappa}}{2\pi}\,\mathrm{Vol}(M).
		\]
	\end{itemize}
\end{teorintro}

\noindent Stable maximal hypersurfaces have been previously studied in \cite{BF} and \cite{F}, where the authors introduced the relative variational formulas and some characterizations in certain ambient spacetimes. More recently, sufficient conditions to ensure stability, in some physically relevant spacetimes, have been given in \cite{FRS}. In fact, a maximal hypersurface $\psi:M\longrightarrow\overline{M}$, with unit normal vector $N$ and shape operator $A$, is stable if and only if the differential operator $L = \Delta - \overline{\mathrm{Ric}}(N,N) - \mathrm{trace}(A^2)$ has non-negative first eigenvalue on $M$. More generally, a maximal hypersurface is said to have finite index if the stability operator $L$ has finite index. In a $3$-dimensional spacetime with nonnegative Ricci curvature on spacelike directions, we prove that any complete maximal surface with finite index has either finite or positive infinite total curvature, provided it is well defined. We remark that when $M$ is an oriented surface with Gaussian curvature $\mathrm{K}$, its total curvature is defined as
	$$
		\int_M \mathrm{K} = \int_M \mathrm{K}_+ - \int_M \mathrm{K}_-,
	$$
where $\mathrm{K}_+$ and $\mathrm{K}_-$ are the positive and negative parts of $\mathrm{K}$. Hence, when $M$ is noncompact, the total curvature is well defined only if at least one of the integrals on the right side is finite.

\begin{teorintro}
	\label{intro:teosup2}
	Let $\overline{M}$ be a $3$-dimensional spacetime with nonnegative Ricci curvature on spacelike vectors, $\psi:M\longrightarrow\overline{M}$ a complete maximal \tcred{oriented} surface immersed in $\overline{M}$. If $M$ has finite index and its total curvature is well defined, then
	$$
		\int_M \mathrm{K} > -\infty.
	$$
\end{teorintro}

A physically relevant family of Lorentzian manifolds is that of Generalized Robertson-Walker (GRW) spacetimes. They can be defined as product manifolds $\overline{M}=I\times F$, where $I\subseteq\mathbb{R}$ is an open interval with the standard negative definite metric $-dt^2$ and $(F,g_F)$ is a Riemannian manifold. On $\overline{M}$ we put a Lorentzian warped product metric of the form $\overline{g} = -dt^2 + \rho(t)^2g_F$, with $\rho$ a smooth positive function on $I$. In these ambient manifolds we give the following generalization of the first part of Theorem \ref{intro:teosup1}, suggested by the work of Albujer and Al\'ias, \cite{AA}, where the authors consider maximal surfaces in Lorentzian products, that is, the case $\rho\equiv1$.

\begin{teorintro}
	\label{intro:teosup3}
	Let $\overline{M}$ be a $3$-dimensional GRW spacetime $\overline{M}=I\times_{\rho}F$ with nonnegative sectional curvatures on spacelike $2$-planes and let $\psi:M\longrightarrow\overline{M}$ be a complete maximal surface. Then $\psi$ is totally geodesic and one of the following cases occurs:
	\begin{itemize}[noitemsep,nolistsep]
		\item [i)] $\psi(M)$ is a spacelike slice $\{t_0\}\times F$ for some $t_0\in I$ such that $\rho'(t_0)=0$,
		\item [ii)] $F$ is a Riemann surface with a complete, flat metric $g_F$ and $\overline{M}$ is the product manifold $\mathbb{R}\times F$ with the flat metric $-dt^2 + g_F$,
		\item [iii)] $F$ is a compact Riemann surface with a metric $g_F$ of constant positive Gaussian curvature, $M$ is a round sphere and the spacetime $\overline{M}$ has constant positive curvature in the smallest slab $I'\times F\subseteq\overline{M}$, $I'\subseteq I$, such that $\psi(M)\subseteq I'\times F$.
	\end{itemize}
\end{teorintro}

\noindent In a GRW spacetime with sectional curvatures bounded below, the image of a  complete maximal hypersurface whose projection on the $I$ factor is relatively compact in $I$ must always intersect at least one totally geodesic spacelike slice of the ambient spacetime. More precisely, we have the following

\begin{teorintro}
	\label{intro:teoGRW1}
	Let $\overline{M} = I\times_{\rho}F$ be a GRW spacetime whose sectional curvatures on spacelike $2$-planes are bounded below. Let $\psi:M\longrightarrow\overline{M}$ be a complete maximal hypersurface, $I'\subseteq I$ the smallest interval \tcred{closed in $I$} such that $\psi(M)\subseteq I'\times F$. If $I'=[t_{\ast},t^{\ast}]$ for some $t_{\ast},t^{\ast}\in I$, then $\rho'(t^{\ast})\geq0$ and $\rho'(t_{\ast})\leq0$. In particular, if $\psi(M)$ is contained in a slab $[a,b]\times F$, $a,b\in I$, then there exists $t_0\in[a,b]$ such that $\rho'(t_0)=0$ and
	\begin{itemize}[noitemsep,nolistsep]
		\item [i)] if $\rho'<0$ on $[a,t_0)$ and $\rho'>0$ on $(t_0,b]$, then $\psi(M)$ must intersect the spacelike slice $\{t_0\}\times F$,
		\item [ii)] if $\rho'>0$ on $[a,t_0)$ and $\rho'<0$ on $(t_0,b]$, then $\psi(M) = \{t_0\}\times F$.
	\end{itemize}
\end{teorintro}

\noindent As an application of Theorems \ref{intro:teoGRW1} and \ref{intro:teocsc}, we give a simple proof of the following Frankel type result.

\begin{corointro}
	\label{intro:corodS}
	Let $\overline{M}=\mathbb{S}^{m+1}_1(\overline{\kappa})$ be the $(m+1)$-dimensional de Sitter spacetime of constant curvature $\overline{\kappa}>0$ and let $\psi_0 : M_0\to\overline{M}$, $\psi_1 : M_1\to\overline{M}$ be two complete maximal hypersurfaces. Then $\psi_0(M_0) \cap \psi_1(M_1) \neq \varnothing$.
\end{corointro}

\noindent If $\rho$ satisfies $\rho''\leq0$ on $I$, then every maximal hypersurface in $\overline{M}=I\times_{\rho}F$ is stable, as observed in Theorem 9 of \cite{FRS}. On the other hand, if a compact maximal hypersurface $M$ is stable in $\overline{M}$ and $\rho''\geq0$ on $I$, then $\rho''\equiv0$ on $M$. More precisely, we have the following

\begin{teorintro}
	\label{intro:teoGRW2}
	Let $M$ be a complete \tcred{oriented} stable maximal hypersurface in a GRW spacetime $\overline{M} = I \times_{\rho} F$ and let $I'\subseteq I$ be the smallest interval such that $\psi(M)\subseteq I'\times F$.
	\begin{itemize}[noitemsep,nolistsep]
		\item [i)] If $M$ is compact then either $\rho''\equiv0$ or $\rho''$ attains both positive and negative values on $I'$.
		\item [ii)] If $M$ is non-compact and, for some $o\in M$, the normal vector field $N$ of $M$ satisfies
		$$
			\lim_{r\to+\infty} \int_a^r\left(\int_{\partial_{B_r}} \overline{g}(T,N)^2\right)^{-1} = +\infty,
		$$
		for some (hence any) $a>0$, where $B_r$ is the geodesic ball of $(M,g)$ with radius $r$ centered at $o$ and $T=\rho(t)\partial_t$, then either $\rho''\equiv0$ on $I'$ or there exists $t_0\in I'$ such that $\rho''(t_0)<0$.
	\end{itemize}
\end{teorintro}

In this work we also obtain new results for higher order mean curvatures. In particular, we study the $k$-stability of spacelike hypersurfaces with zero $(k+1)$-th mean curvature in spacetimes of constant curvature. The notion of $k$-stability has been previously studied in the Lorentzian setting in \cite{BC} as well as in \cite{CC}. The following two results are somehow companions of Theorems \ref{intro:teostab}, \ref{intro:teosup2} and \ref{intro:teoGRW1} in this context. Note that the requirements on the sign of the $k$-th mean curvature function $H_k$ and on the rank of the shape operator $A$ are minimal to guarantee the ellipticity of the $k$-stability operator of $\psi$, as defined in Section \ref{sekpre}.

\begin{teorintro}
	\label{intro:kmax1}
	Let $\psi: M \longrightarrow \overline{M}$ be a complete \tcred{oriented} spacelike hypersurface with zero $(k+1)$-th mean curvature, for some $1\leq k\leq m-2$, in a spacetime $\overline{M}$ of constant curvature $\overline{\kappa}$. Suppose that the $k$-th mean curvature function $H_k$ is positive and that the shape operator $A$ has rank $>k$ on $M$.
	\begin{itemize}[noitemsep,nolistsep]
		\item [i)] If $\overline{\kappa}\leq0$, then $M$ is non-compact and $\psi$ is $k$-stable.
		\item [ii)] If $\overline{\kappa}>0$ and $M$ is compact, then $\psi$ is not $k$-stable.
		\item [iii)] If $\overline{\kappa}>0$ and we assume that $M$ is non-compact and that, for some $o\in M$ and for some (hence any) $a\in\mathbb{R}$,
		\begin{align*}
			& \lim_{r\to+\infty} \int_a^r \left( \int_{\partial B_r} H_k \right)^{-1} = +\infty, \\
			& \lim_{r \longrightarrow + \infty} \int_{B_r} \left( \overline{\kappa} \ \mathrm{trace}(P_k) - \mathrm{trace}(A^2 P_k) \right) = + \infty,
		\end{align*}
		where $B_r$ is the geodesic ball of $(M,g)$ with radius $r$ centered at $o$ and $P_k$ is the $k$-th Newton operator associated to $\psi$ (as defined in Section \ref{sekpre}), then for every compact subset $K \subseteq M$ the hypersurface $\psi|_{M\setminus K} : M\setminus K \longrightarrow \overline{M}$ is not $k$-stable.
	\end{itemize}
\end{teorintro}

\noindent We remark that the quantity $\overline{\kappa} \ \mathrm{trace}(P_k) - \mathrm{trace}(A^2 P_k)$ appearing in the statement of Theorem \ref{intro:kmax1} can be expressed in terms of the higher order mean curvature functions $H_1,\dots,H_m$ of $\psi$ as
\begin{equation}
	\label{intro:kmaxrem}
	\begin{split}\overline{\kappa} \ \mathrm{trace}(P_k) - \mathrm{trace}(A^2 P_k) & = \binom{m}{k}(m-k)\overline{\kappa}H_k \\
	& \phantom{= \;} + \binom{m}{k+1}\left((m-k-1)H_{k+2} - m H_1H_{k+1}\right).
	\end{split}
\end{equation}

\begin{teorintro}
	\label{intro:kmax2}
	Let $\overline{M}= I \times_\rho F$ be a Robertson-Walker spacetime of constant sectional curvature and let $\psi: M \longrightarrow \overline{M}$ be a complete \tcred{oriented} spacelike hypersurface with zero $(k+1)$-th mean curvature. Suppose that there exists $C>0$ such that $C^{-1} < H_k < C$, that $\mathrm{rank}(A) > k$ on $M$ and that for some $o\in M$ and for some (hence any) $a\in\mathbb{R}$, one of the following conditions is satisfied:
	\begin{alignat*}{4}
		(i) & \qquad \lim_{r\to+\infty} \int_a^r \frac{dt}{\mathrm{Vol}(\partial B_r)} = +\infty && \qquad \text{and} \qquad \mathrm{Vol}(M) = +\infty, && \qquad \text{or} \\
		(ii) & \qquad \lim_{r\to+\infty} \int_a^r \frac{dt}{\mathrm{Vol}(\partial B_r)} < +\infty && \qquad \text{and} \qquad \lim_{r\to+\infty} \frac{\log(\mathrm{Vol}(\partial B_r))}{r} = 0,
	\end{alignat*}
	where $B_r$ is the geodesic ball of $(M,g)$ with radius $r$ centered at $o$. If $\psi(M)$ is contained in a slab $[a, b] \times F$, then there exists $t_0\in[a,b]$ such that $\rho'(t_0) = 0$ and $\psi(M)$ must intersect the spacelike slice $\{t_0\}\times F$.
\end{teorintro}

\bigskip

Our paper is organized as follows. In Section \ref{sepre} we recall the notion of stability and $k$-stability for spacelike hypersurfaces with vanishing mean curvature functions together with some general properties of Schr\"odinger differential operators that we will need in the subsequent sections. In Section 3 we mainly deal with maximal hypersurfaces in locally symmetric spacetimes and we prove Theorems \ref{intro:teostab}, \ref{intro:teocsc}, \ref{intro:teosup1} and \ref{intro:teosup2} above (see Theorems \ref{teocpt}, \ref{teocsc}, \ref{teosup1} and \ref{teoindex1}, respectively). In Section \ref{segrw} we study maximal hypersurfaces in GRW spacetimes and we prove Theorems \ref{intro:teosup3}, \ref{intro:teoGRW1}, \ref{intro:teoGRW2} and Corollary \ref{intro:corodS}  (see Theorems \ref{teosup3}, \ref{teoGRW2} and Corollaries \ref{teoGRW1}, \ref{corodS}), also giving a characterization of GRW spacetimes with spacelike sectional curvatures bounded below (Lemma \ref{charsectbound}). In section \ref{sho} we consider hypersurfaces with zero $(k+1)$-th order mean curvature in spacetimes of constant curvature and we prove Theorems \ref{intro:kmax1} and \ref{intro:kmax2} (see Theorems \ref{teoks} and \ref{kmax2}).

\section{Preliminaries} \label{sepre}

We devote this section to introduce the basic concepts concerning the stability of maximal hypersurfaces in general ambient spacetimes as well as their natural generalization to the case of $k$-stable spacelike hypersurfaces with zero $(k+1)$-th mean curvature.

\subsection{Stability of maximal hypersurfaces}

\label{sesta}

Let $\psi: M \longrightarrow \overline{M}$ be a spacelike hypersurface immersed in a spacetime $(\overline{M}, \overline{g})$ and let $g=\psi^{\ast}\overline{g}$ be the Riemannian metric induced on $M$. We will denote by $N$ a chosen unit normal vector field to $M$ and by $A$ the shape operator in the direction of $N$, determined by the validity of the Weingarten formula $AX = -\overline{\nabla}_X N$ for any $X\in TM$, with $\overline{\nabla}$ the Levi-Civita connection of $(\overline{M},\overline{g})$. The linear operator $A$ is self-adjoint at each tangent space and its eigenvalues $\kappa_1, \dots, \kappa_m$ are, by definition, the principal curvatures of the hypersurface. The mean curvature function $H$ of $\psi$ in the direction of $N$ is given by the normalized trace of $-A$, that is,
\begin{equation}
	\label{Hf}
	\tcr{H = -\frac{1}{m}\mathrm{trace}(A) = -\frac{\kappa_1 + \cdots + \kappa_m}{m}.}
\end{equation}

A spacelike hypersurface is called a maximal hypersurface if it is a critical point of the volume functional for compactly supported normal variations of the immersion. This condition is equivalent to the hypersurface having zero mean curvature. If $\psi$ is a maximal hypersurface and we take a normal variation driven by a variational vector field $\phi N$, with $\phi$ a smooth function supported in a relatively compact domain $\Omega\subseteq M$, then the second variation of the volume of $\Omega$ is (see Theorem 2.1 of \cite{BF} and Theorem 1 of \cite{F})
\begin{equation}
	\label{2var}
	\int_{\Omega} \left( \Delta \phi - \left( \overline{\mathrm{Ric}}(N, N) + \mathrm{trace}(A^2) \right) \phi \right) \phi \ dV,
\end{equation}
where $\nabla$ is the Levi-Civita connection of $(M,g)$, $\Delta$ is the corresponding Laplacian and $\overline{\mathrm{Ric}}$ is the Ricci tensor of $(\overline{M},\overline{g})$. In view of the identity
\begin{equation}
	\label{scalmax}
	\mathrm{S} = \overline{\mathrm{S}} + 2\overline{\mathrm{Ric}}(N,N) + \mathrm{trace}(A^2)
\end{equation}
satisfied by the scalar curvatures $\mathrm{S}$, $\overline{\mathrm{S}}$ of $M$ and $\overline{M}$ (see formula (\ref{scal}) in Section \ref{gausseq}), we have
\begin{equation}
	\label{stab}
	\overline{\mathrm{Ric}}(N, N) + \mathrm{trace}(A^2) = \mathrm{S} - \left( \overline{\mathrm{S}} + \overline{\mathrm{Ric}}(N,N) \right) = \sum_{i=1}^m\left(\mathrm{Ric}(E_i,E_i)-\overline{\mathrm{Ric}}(E_i,E_i)\right)
\end{equation}
for any local orthonormal frame $\{E_1,\dots,E_m\}$ on $TM$, where $\mathrm{Ric}$ is the Ricci tensor of $(M,g)$. The immersion $\psi$ is said to be stable if the second variation of the volume of $M$ is non-positive for every compactly supported normal variation. Stability is detected by the sign of the first eigenvalue $\lambda_1^L(M)$ of the stability operator $L$ defined by
\begin{equation}
	\label{lu}
	L u = \Delta u - \left( \overline{\mathrm{Ric}}(N,N) + \mathrm{trace}(A^2) \right) u \qquad \text{for every } u\in C^2(M),
\end{equation}
analogously to what happens in the Riemannian case for minimal hypersurfaces (see for instance \cite{FS, JS}). In fact, since $\lambda_1^L(M)$ is variationally characterized by
\begin{equation}
	\label{lambdastab}
	\lambda_1^L(M) = \inf_{\substack{\phi \in C_c^\infty (M) \\ \ \phi \neq 0}} \frac{\int_M |\nabla\phi|^2 + \left( \overline{\mathrm{Ric}}(N,N) + \mathrm{trace}(A^2) \right) \phi^2}{\int_M \phi^2 },
\end{equation}
by applying the divergence theorem to (\ref{2var}) we see that $\psi$ is stable if and only if $\lambda_1^L(M)\geq0$. More generally, $M$ is said to have finite index if the operator $L$ has finite index. When $M$ is compact, this is always the case, while for a complete, non-compact hypersurface this happens if and only if there exists a relatively compact open set $\Omega\subseteq M$ such that the second variation of the volume of $M$ is non-positive for every normal variation compactly supported in $M\setminus\overline{\Omega}$. In this case, we also say that $M$ is stable at infinity.

\subsection{$k$-stability of spacelike hypersurfaces with zero $(k+1)$-th mean curvature}

\label{sekpre}

We can generalize the concepts above to study the $k$-stability of spacelike hypersurfaces with zero $(k+1)$-th mean curvature in spacetimes of constant curvature. In order to do so, let $\psi: M \longrightarrow \overline{M}$ be a spacelike hypersurface in a spacetime $\overline{M}$ with constant curvature $\overline{\kappa}$ and let $A$ be the shape operator of the immersion with respect to a unit normal timelike vector $N$. We associate to $A$ the algebraic invariants $S_1,\dots,S_m$ and the mean curvature functions $H_1,\dots,H_m$ of orders $1,\dots,m$ in the direction of $N$ by setting
\begin{equation}
	\label{Ak}
	S_k = \sum_{1\leq i_1<\cdots<i_k\leq m} \kappa_{i_1}\cdots \kappa_{i_m}, \qquad H_k = \frac{(-1)^k}{\binom{m}{k}}S_k \qquad \text{for } 1\leq k\leq m.
\end{equation}
The Newton tensors $P_k: TM \longrightarrow TM$, $0\leq k\leq m$, are inductively defined by
\begin{equation}
	\label{newton}
	P_0 = \mathbb{I} \qquad \text{and} \qquad \ P_k = \binom{m}{k} H_k \mathbb{I} + A \circ P_{k-1} \qquad \text{for } 1\leq k\leq m,
\end{equation}
where $\mathbb{I}$ denotes the identity on $TM$, and they satisfy the following identities, proved in \cite{ABC}.

\begin{lema}
	\label{newtop}
	For $0\leq k\leq m-1$, let $c_k=(m-k)\binom{m}{k}$. Then, for every $X\in\mathfrak X(M)$
	\begin{equation}
		\label{eq:newtop}
		\begin{split}
		& \mathrm{trace}(P_k) = c_k H_k, \quad \mathrm{trace}(P_k\circ A) = -c_k H_{k+1}, \\
		& \mathrm{trace}(P_k\circ\nabla_X A) = -\binom{m}{k+1} \nabla_X H_{k+1}.
		\end{split}
	\end{equation}
\end{lema}

\noindent Note that the identity (\ref{intro:kmaxrem}) stated in the Introduction follows by applying relations (\ref{newton}) and (\ref{eq:newtop}). For $0\leq k\leq m$, let $L_k$ be the second order linear differential operator given by
\begin{equation}
	\label{Lkdef}
	L_k u = \mathrm{trace} \left( P_k \circ \mathrm{Hess}(u) \right) \equiv \mathrm{div}(P_k(\nabla u)) \qquad\text{for every } u\in C^2(M),
\end{equation}
where the second equality holds because $P_k$ is divergence-free as long as $\overline{M}$ has constant curvature. Note that $H_1=H$ and that $L_0=\Delta$. The operator $L_k$ is elliptic if and only if $P_k$ is positive definite. For $1\leq k\leq m-2$, if $H_{k+1}=0$ at some point then $P_k$ is positive definite there if and only if $H_k>0$ and $\mathrm{rank}(A)>k$, see Corollary 2.3 of \cite{HL}.

Now, for $0\leq k\leq m-1$ we can define the $k$-volume functional for a relatively compact domain $\Omega \subset M$ by setting
$$
	\mathrm{Vol}_k(\Omega) = \int_\Omega F_k( S_1, S_2, \dots , S_k) dV,
$$
where $S_1,\dots,S_m$ are the invariants defined above and the functions $F_0,\dots, F_{m-1}$ are recursively defined by
$$
	F_0 = 1, \qquad F_1 = -S_1, \qquad F_k = (-1)^k S_k -\frac{\overline{\kappa}(m - k + 1)}{k-1} F_{k-2} \qquad \text{for } 2\leq k\leq m-1.
$$
According to \cite{CC}, if $M$ has zero $(k+1)$-th mean curvature for some $1\leq k\leq m-1$ and we take a normal variation of $M$ given by $\phi N$, with $\phi$ supported in a relatively compact domain $\Omega\subseteq M$, then the second variation of the $k$-volume functional of $M$ is
\begin{equation}
	\label{k2var}
	(k+1) \int_\Omega \left( L_k \phi + \overline{\kappa} \ \mathrm{trace}(P_k) \phi - \mathrm{trace}(A^2 P_k) \phi \right) \phi \ dV.
\end{equation}
As a generalization of (\ref{stab}), for every local orthonormal frame $\{E_1,\dots,E_m\}$ on $TM$ we have
$$
	- \overline{\kappa} \ \mathrm{trace}(P_k) + \mathrm{trace}(A^2 P_k) = \sum_{i=1}^m \left( \mathrm{Ric}(P_k E_i,E_i)-\overline{\mathrm{Ric}}(P_k E_i,E_i) \right),
$$
as a consequence of formula (\ref{Pscal}) of Section \ref{gausseq} and of Lemma \ref{newtop}. The immersion $\psi$ is said to be $k$-stable if the second variation of the $k$-volume of $M$ is non-positive for every compactly supported normal variation. Analogously to the maximal case, $k$-stability is detected by the sign of the first eigenvalue $\lambda_1^{\widetilde{L}_k} (M)$ of the $k$-stability operator $\widetilde{L}_k$ defined by
\begin{equation}
	\label{kst}
	\widetilde{L}_k u = L_k u - \left( \mathrm{trace}(A^2 P_k) - \overline{\kappa} \ \mathrm{trace}(P_k) \right) u \qquad \text{for every } u\in C^2(M).
\end{equation}
Since $\lambda_1^{\widetilde{L}_k} (M)$ is variationally characterized by
\begin{equation}
	\label{kst1}
	\lambda_1^{\widetilde{L}_k} (M) = \inf_{\substack{\phi \in C_c^\infty (M) \\ \ \phi \neq 0}} \frac{\int_M g(P_k(\nabla\phi),\nabla\phi) + \left( \mathrm{trace}(A^2 P_k) - \overline{\kappa} \ \mathrm{trace}(P_k) \right) \phi^2}{\int_M \phi^2 },
\end{equation}
we see that a spacelike hypersurface $M$ with zero $(k+1)$-th mean curvature is $k$-stable if and only if $\lambda_1^{\widetilde{L}_k} (M) \geq 0$. Similarly to what happens with the usual stability operator, a complete, non-compact, spacelike hypersurface with zero $(k+1)$-th mean curvature is said to be $k$-stable at infinity if the $k$-stability operator has finite index.

\subsection{General facts on Schr\"odinger operators}

\label{schrodinger}

Consider a Riemannian manifold $(M,g)$, a function $F(x)\in L^{\infty}_{\mathrm{loc}}(M)$ and let $P : TM \longrightarrow TM$ be a positive definite, self-adjoint, endomorphism of class $C^1$. Define the second order linear elliptic operator $L$ by setting
\begin{equation}
	\label{defL}
	Lu = \mathrm{div}(P(\nabla u)) - F(x)u \qquad \text{for every } u\in C^2(M).
\end{equation}
For every open set $\Omega \subseteq M$, let $\lambda_1^L(\Omega)$ be the first eigenvalue of $L$ on $\Omega$, given by
\begin{equation}
	\label{lambda1}
	\lambda_1^L(\Omega) = \inf_{\substack{\phi \in C_c^\infty (\Omega) \\ \ \phi \neq 0}} \frac{\int_\Omega g(P(\nabla\phi),\nabla\phi) + F(x) \phi^2}{\int_\Omega \phi^2 } = \inf_{\substack{\phi \in C_c^\infty (\Omega) \\ \ \phi \neq 0}} \frac{\int_\Omega -\phi L\phi}{\int_\Omega \phi^2 }
\end{equation}
If $\Omega$ is a relatively compact domain with sufficiently regular boundary, say of class $C^2$, then the infimum in the RHS of (\ref{lambda1}) is achieved by the non-zero solutions of the Dirichlet problem
\begin{alignat*}{2}
	Lu + \lambda_1^L(\Omega) u & = 0 && \quad \text{on } \Omega, \\
	u & = 0 && \quad \text{on } \partial\Omega,
\end{alignat*}
which belong to $C^{0,\alpha}(\overline{\Omega}) \cap H^2(\Omega) \cap C^1(\Omega)$ for some $0<\alpha<1$ (see Theorems 8.6, 8.12, 8.29 of \cite{GT} and Theorem 1 of \cite{T}).
For operators of the form (\ref{defL}), we have the next monotonicity property of the eigenvalues with respect to the domain.

\begin{prop}
	\label{propo2}
	Let $(M,g)$ be a Riemannian manifold and for $F(x) \in L^{\infty}_{\mathrm{loc}}(M)$, $P : TM \longrightarrow TM$ let $L$ be the operator defined in (\ref{defL}). Let $\Omega_1, \Omega_2$ be two relatively compact domains in $M$ such that $\Omega_1 \subseteq \Omega_2$. Then
	\begin{equation}
		\label{monotonicity}
		\lambda_1^L (\Omega_1) \geq \lambda_1^L (\Omega_2).
	\end{equation}
	If $\Omega_1$ and $\Omega_2$ have $C^2$ boundaries and the interior of $\Omega_2 \setminus \Omega_1$ is nonempty, then (\ref{monotonicity}) holds with strict inequality sign.
\end{prop}

\begin{demo}
	Observe that (\ref{monotonicity}) is a trivial consequence of the definition (\ref{lambda1}). To prove the last statement we will proceed as follows. Consider an open subset $\Omega\subseteq M$ and two functions $u, v \in C^1(\Omega)$, with $v \neq 0$ on $\Omega$: since $P$ is \tcr{positive definite} and self-adjoint, 
	a direct computation yields the following extension of the classic Picone's identity
	\begin{equation}
		\label{picone}
		0 \leq g\left( P\left( \nabla u - \frac{u}{v} \nabla v \right), \nabla u - \frac{u}{v} \nabla v \right) =  g(P(\nabla u),\nabla u) - g \left( \nabla \left(\frac{u^2}{v} \right), P (\nabla v) \right).
	\end{equation}
	In particular,
	$$
		g\left( P\left( \nabla u - \frac{u}{v} \nabla v \right), \nabla u - \frac{u}{v} \nabla v \right) \equiv 0 \qquad \text{on } \Omega
	$$
	if and only if $u = C v$ for some constant $C \in \mathbb{R}$. Now, let us suppose that $\Omega_1$ and $\Omega_2$ have $C^2$ boundaries and let $u$ and $v$ be non-zero solutions of the Dirichlet problems
	\begin{alignat*}{6}
		Lu + \lambda_1^L(\Omega_1) u & = 0 && \quad \text{in } \Omega_1, & \qquad \qquad && Lv + \lambda_1^L(\Omega_2) v & = 0 && \quad \text{in } \Omega_2, \\
		u & = 0 && \quad \text{on } \partial\Omega_1, & \qquad \qquad && v & = 0 && \quad \text{on } \partial\Omega_2.
	\end{alignat*}
	Note that we can suppose $v >0$ on $\Omega_2$. Taking (\ref{picone}) into account, since $u=0$ on $\partial\Omega_1$ we get
	\begin{align*}
		0 & \leq \int_{\Omega_1} g(P(\nabla u),\nabla u) - g\left( \nabla \left(\frac{u^2}{v} \right), P(\nabla v) \right) = \left( \lambda_1^L(\Omega_1) - \lambda_1^L(\Omega_2) \right) \int_{\Omega_1} u^2.
	\end{align*}
	We now reason by contradiction assuming that $\overset{\circ}{\Omega_2 \setminus \Omega_1} \neq \varnothing$ and $\lambda_1^L(\Omega_1) = \lambda_1^L(\Omega_2)$. From the above inequalities it follows that on the connected components of $\Omega_1$, $u = C v$ for some $C \in \mathbb{R}$. Choose one of the component, say $\widetilde{\Omega}_1$, with $\partial \widetilde{\Omega}_1 \cap \Omega_2 \neq \varnothing$ (this is possible since $\Omega_2 \setminus \Omega_1$ has non-empty interior). Since $u \equiv 0$ on $\partial \Omega_1$, we have $v \equiv 0$ on $\partial \widetilde{\Omega}_1 \cap \Omega_2 \neq \varnothing$, reaching a contradiction.  
\end{demo}

The following generalization of Barta's theorem also holds.

\begin{prop}
	\label{barta}
	Let $(M,g)$ be a Riemannian manifold and for $F(x) \in L^{\infty}_{\mathrm{loc}}(M)$, $P : TM \longrightarrow TM$ let $L$ be the operator defined in (\ref{defL}). If $u\in C^2(M)$ is a positive function, we have
	\begin{equation}
		\label{bartaineq}
		\lambda_1^L(M) \geq \inf_M \left( -\frac{Lu}{u} \right).
	\end{equation}
\end{prop}

\begin{demo}
	Let $\phi \in C_c^\infty (M)$ and consider the vector field
	$$
		Z = -\frac{\phi^2}{u} P(\nabla u).
	$$
	Taking the divergence and integrating we get
	\begin{equation}
		\label{barta1}
		2 \int_M \frac{\phi}{v} g\left( P(\nabla u), \nabla\phi \right) = - \int_M \phi^2 \mathrm{div}\left( \frac{P(\nabla u)}{u} \right).
	\end{equation}
	Since $P$ is positive definite and self-adjoint, by Cauchy-Schwarz and Young's inequalities
	\begin{equation}
		\label{barta2}
		\begin{split}
			2 \frac{\phi}{u} g\left( P(\nabla u), \nabla\phi \right) &  \leq 2 \frac{|\phi|}{u}\sqrt{g(P(\nabla u),\nabla u)}\sqrt{g(P(\nabla\phi),\nabla\phi)} \\
			& \leq \frac{\phi^2}{u^2} g(P(\nabla u),\nabla u) + g(P(\nabla\phi),\nabla\phi).
		\end{split}
	\end{equation}
	Using (\ref{barta2}), (\ref{barta1}) and the divergence theorem again we obtain
	\begin{equation}
		\label{bar2}
		\begin{split}
			\int_M -\phi L\phi & = \int_M g(P(\nabla\phi),\nabla\phi) + \int_M F(x)\phi^2 \\
			& \geq -\int_M \frac{\phi^2}{u^2} g(P(\nabla u),\nabla u) + 2 \int_M \frac{\phi}{v} g\left( P(\nabla u), \nabla\phi \right) + \int_M F(x)\phi^2 \\
			& = - \int_M \left( g\left(\frac{\nabla u}{u^2}, P(\nabla u)\right) + \mathrm{div}\left( \frac{P(\nabla u)}{u} \right) - F(x) \right) \phi^2 \\
			& = - \int_M \left( \frac{\mathrm{div}( P(\nabla u))}{u} - F(x) \right) \phi^2 = \int_M -\frac{Lv}{v} \phi^2 \geq \inf_M \left(- \frac{Lv}{v} \right) \int_M \varphi^2 
		\end{split}
	\end{equation}
	and by definition (\ref{lambda1}) of $\lambda_1^L(M)$ we obtain inequality (\ref{bartaineq}).
\end{demo}

We conclude this paragraph with the following characterization of non-negativity of $\lambda_1^L(\Omega)$ for an open subset $\Omega \subseteq M$. For $P = \mathbb{I} : TM \longrightarrow TM$ and $F(x) \in C^{\infty}(M)$, it is given as Theorem 1 in \cite{FS}. For $P = \mathbb{I}$ and $F(x) \in L^{\infty}_{\mathrm{loc}}(M)$ it is proved as Lemma 3.10 of \cite{PRS2}. The proof for a general self-adjoint, positive definite $C^1$ endomorphism $P$ is a straightforward extension of the proof given in \cite{PRS2} for $P = \mathbb{I}$.

\begin{lema}
	\label{lemeq}
	Let $(M,g)$ be a Riemannian manifold, $\Omega \subseteq M$ an open set with possibly non-compact closure. For $F(x)\in L^{\infty}_{\mathrm{loc}}(M)$, $P : TM \longrightarrow TM$ as above, let $\lambda_1^L(\Omega)$ be the first eigenvalue on $\Omega$ of the operator $L$ defined in (\ref{defL}). Then, the following conditions are equivalent:
	\begin{enumerate}[noitemsep,nolistsep]
		\item $\lambda_1^L (\Omega) \geq 0$.
		\item There exists $u \in C^1(\Omega)$, $u > 0$, weak solution of $Lu =0$ on $\Omega$.
		\item There exists $u \in H^1_{\mathrm{loc}}(\Omega)$, $u > 0$, weak solution of $Lu \leq 0$ on $\Omega$.
	\end{enumerate}
\end{lema}

\noindent We remark that when $P$ and $F(x)$ are smooth, which will always be the case in the following, standard elliptic regularity ensures that solutions of $Lu=0$ are also smooth.

\subsection{Gauss equations}

\label{gausseq}

Our main reference is O'Neill's book \cite{O'N}. However, we remark that we adopt the convention of defining the Riemann curvature operator $\mathrm{R}$ of a semi-Riemannian manifold $(M,g)$ by setting
$$
	R(X,Y)Z = \nabla_X(\nabla_Y Z) - \nabla_Y(\nabla_X Z) - \nabla_{[X,Y]}Z \qquad \text{for every } X,Y,Z\in\mathfrak X(M),
$$
so we have $R(X,Y)Z = -R_{XY}Z = R_{YX}Z$ for every $X,Y,Z\in\mathfrak X(M)$, where $R_{(\;\cdot\;,\;\cdot\;)}(\;\cdot\;)$ is the notation used in \cite{O'N}. The $(0,4)$-form Riemann curvature tensor $\mathrm{Riem}$ is then given by
$$
	\mathrm{Riem}(W,Z,X,Y) = g(R(X,Y)Z,W) \qquad \text{for every } X,Y,Z,W\in \mathfrak X(M),
$$
the sectional curvature of any non-degenerate $2$-plane $X\wedge Y\subseteq TM$ spanned by a couple of vectors $X,Y\in TM$ is
\begin{equation}
	\label{sect}
	\mathrm{Sect}(X\wedge Y) = \frac{\mathrm{Riem}(X,Y,X,Y)}{g(X,X)g(Y,Y)-g(X,Y)^2}.
\end{equation}
For every $p\in M$ and for every choice of a $g$-orthonormal basis $\{E_i,\dots,E_{\dim M}\}$ of $T_p M$, the values of the Ricci tensor $\mathrm{Ric}$ and of the scalar curvature $\mathrm{S}$ are given by
\begin{align}
	\label{ricci}
	\mathrm{Ric}(X,Y) & = \sum_{i=1}^{\dim M} g(E_i,E_i)\mathrm{Riem}(X,E_i,Y,E_i) \qquad \text{for every } X,Y\in T_p M, \\
	\label{scalar}
	\mathrm{S}(p) & = \sum_{i=1}^{\dim M} g(E_i,E_i)\mathrm{Ric}(E_i,E_i).
\end{align}

Let $\psi:M\longrightarrow\overline{M}$ be a spacelike hypersurface immersed in a spacetime $\overline{M}$, with unit timelike vector field $N$ and shape operator $A$ in the direction of $N$. For every $X,Y,Z,W\in \mathfrak X(M)$ we have the validity of Gauss equations (see Theorem 4.5 and Lemma 4.19 of \cite{O'N})
\begin{equation}
	\label{gauss}
	\mathrm{Riem}(X,Z,Y,W) = \overline{\mathrm{Riem}}(X,Z,Y,W) - g(AX,Y) g(AZ,W) + g(AX,W) g(AZ,Y),
\end{equation}
for any $X,Y,Z,W\in TM$, where $\mathrm{Riem}$ and $\overline{\mathrm{Riem}}$ are the Riemann curvature tensors of $(M,g)=(M,\psi^{\ast}\overline{g})$ and $(\overline{M},\overline{g})$, respectively. For every $p\in M$, $X,Z\in T_p M$ and for any orthonormal basis $\{E_1,\dots,E_m\}$ of $T_p M$ we have, by (\ref{ricci}) and (\ref{gauss}),
\begin{equation}
	\label{ric}
	\begin{split}
		\mathrm{Ric}(X,Y) & = \sum_{i=1}^m \overline{\mathrm{Riem}}(X,E_i,Y,E_i) - \sum_{i=1}^m g(AX,Y)g(AE_i,E_i) \\
		& \phantom{= \;} + \sum_{i=1}^m g(AX,E_i)g(AY,E_i) \\
		& = \overline{\mathrm{Ric}}(X,Y) + \overline{\mathrm{Riem}}(X,N,Y,N) + mH g(AX,Y) + g(AX,AY),
	\end{split}
\end{equation}
recalling that $mH = -\mathrm{trace}(A)$. Since $\{E_1,\dots,E_m,N\}$ is a $\overline{g}$-orthonormal basis of $T_{\psi(p)}\overline{M}$, by (\ref{scalar}), (\ref{ricci}) and (\ref{ric}) we get
\begin{equation}
	\label{scal}
	\begin{split}
		\mathrm{S} & = \sum_{i=1}^m \left( \overline{\mathrm{Ric}}(E_i,E_i) + \overline{\mathrm{Riem}}(N,E_i,N,E_i) + mH g(AE_i,E_i) + g(A^2E_i,E_i) \right) \\
		& = \sum_{i=1}^m \overline{\mathrm{Ric}}(E_i,E_i) + \overline{\mathrm{Ric}}(N,N) - m^2H^2 + \mathrm{trace}(A^2) \\
		& = \overline{\mathrm{S}} + 2\overline{\mathrm{Ric}}(N,N) - m^2H^2 + \mathrm{trace}(A^2).
	\end{split}
\end{equation}
If $\overline{M}$ has constant sectional curvature $\overline{\kappa}$, then, for every $X,Y,Z,W\in\mathfrak X(\overline{M})$, 
\begin{alignat*}{1}
	\overline{\mathrm{Riem}}(X,Z,Y,W) & = \overline{\kappa}(g(X,Y)g(Z,W)-g(X,W)g(Y,Z)), \\
	\overline{\mathrm{Ric}}(X,Y) & = m\overline{\kappa}\,\overline{g}(X,Y), \\
	\overline{\mathrm{S}} & = m(m+1)\overline{\kappa}
\end{alignat*}
and for every self-adjoint endomorphism $P : TM \longrightarrow TM$ it follows that
\begin{equation}
	\label{Pscal}
	\begin{split}
		\sum_{i=1}^m \mathrm{Ric}(P E_i,E_i) & = \sum_{i=1}^m \overline{\mathrm{Ric}}(P E_i,E_i) + \sum_{i=1}^m \overline{\kappa}\,\overline{g}(N,N) g(PE_i,E_i) \\
		& \phantom{= \;} + \sum_{i=1}^m mHg(APE_i,E_i) + \sum_{i=1}^m g(AP E_i,AE_i) \\
		& = \sum_{i=1}^m \overline{\mathrm{Ric}}(P E_i,E_i) - \overline{\kappa}\,\mathrm{trace}(P) + mH\,\mathrm{trace}(AP) + \mathrm{trace}(A^2P).
	\end{split}
\end{equation}

\section{Maximal hypersurfaces in locally symmetric spacetimes}

In this section we first prove Theorems \ref{intro:teostab} and \ref{intro:teocsc} of the Introduction, then we focus our attention to the case of maximal surfaces and we give proofs of Theorems \ref{intro:teosup1} and \ref{intro:teosup2}. We start with a slight generalization of results obtained by Nishikawa (Theorem B of \cite{N}) and Ishihara (Theorems 1.1 and 1.2 of \cite{I}), whose proof relies on Theorem \ref{KO} below, a consequence of a more general result which is proved as Theorem 3.6 in \cite{AMR}.

\begin{defi}[Definitions 2.1 and 2.3 of \cite{AMR}]
	\label{MP}
	Let $M$ be a Riemannian manifold. The Omori-Yau maximum principle for the Laplacian is said to hold on $M$ if, for any function $u\in C^2(M)$ with $u^{\ast} = \sup_M u < +\infty$, there exists a sequence of points $\{x_k\}_{k\in\mathbb{N}}\subseteq M$ satisfying
	$$
		\text{(i) } u(x_k) > u^{\ast} - \frac{1}{k}, \quad \text{(ii) } |\nabla u(x_k)| < \frac{1}{k} \quad \text{and} \quad \text{(iii) } \Delta u(x_k) < \frac{1}{k} \qquad \forall k\in\mathbb{N}.
	$$
	The weak maximum principle for the Laplacian is said to hold on $M$ if, for any function $u$ as above, there exists a sequence of points $\{x_k\}_{k\in\mathbb{N}}\subseteq M$ such that (i) and (iii) hold.
\end{defi}

\begin{teor}
	\label{KO}
	Let $(M,g)$ be a Riemannian manifold on which the Omori-Yau maximum principle for the Laplacian holds, $a\in\mathbb R$ and $F$ a positive continuous function on $[a,+\infty)$ satisfying
	$$
		\int_{a+\varepsilon}^{+\infty} \left(\int_a^t F(s)ds \right)^{-1/2} dt < +\infty \qquad \text{and} \qquad \limsup_{t\to+\infty} \frac{1}{tF(t)}\int_a^t F(s)ds < +\infty
	$$
	for some (hence, any) $\varepsilon>0$. If $u\in C^2(M)$, $f\in C^0(\mathbb R)$ are such that
	$$
		\Delta u \geq f(u) \quad \text{on } M \qquad \text{and} \qquad \liminf_{t\to+\infty} \frac{f(t)}{F(t)} > 0,
	$$
	then $u^{\ast} := \sup_M u$ is finite and $f(u^{\ast}) \leq 0$.
\end{teor}

\noindent When $M$ is a complete Riemannian manifold, a sufficient condition for the validity of the Omori-Yau maximum principle for the Laplacian on $M$ is the existence of a constant $C\in\mathbb R$ such that $\mathrm{Ric}(X,X)\geq C|X|^2$ for every $X\in TM$, see the book \cite{AMR}. We are now ready to state and prove the following result. Recall that a semi-Riemannian manifold is said to be locally symmetric if its Riemannian curvature tensor is parallel. Semi-Riemannian manifolds of constant curvature provide the simplest examples of such manifolds.

\begin{teor}
	\label{teoni2}
	Let $\overline{M}$ be a locally symmetric spacetime of dimension $m+1$ whose Ricci and sectional curvatures satisfy
	\begin{align}
		\label{ricbound1}
		\overline{\mathrm{Ric}}(Z, Z) \geq c_1 & \qquad \text{for all unit timelike vectors } Z \in T\overline{M} \\
		\label{sectbound1}
		\overline{\mathrm{Sect}}(\Pi) \geq c_2 & \qquad \text{for all non-degenerate tangent $2$-planes } \Pi\subseteq T\overline{M}
	\end{align}
	for some constants $c_1, c_2 \in \mathbb{R}$ and let $\psi : M \longrightarrow \overline{M}$ be a complete maximal hypersurface. Then the shape operator $A$ satisfies
	\begin{equation}
		\label{Abound}
		\mathrm{trace}(A^2) \leq \max\left\{0,-(c_1+2mc_2)\right\} \qquad \text{on } M.
	\end{equation}
	In particular, if $c_1\geq -2mc_2$ then $\psi$ is totally geodesic.
\end{teor}

\begin{demo}
	Let $q\in M$ be a given point and let $\{E_1,\dots,E_m\}$ be an orthonormal basis of $T_q M$ given by the principal directions of curvature, that is, eigenvectors of $A(q)$ corresponding to the principal curvatures $\kappa_1,\dots,\kappa_m$. Since $\overline{M}$ is locally symmetric and $\psi$ is maximal, following Nishikawa \cite{N} we have Simons' formula
	\begin{align*}
		\frac{1}{2}\Delta\mathrm{trace}(A^2) & = |\nabla A|^2 + \overline{\mathrm{Ric}}(N,N)\mathrm{trace}(A^2) \\
		& \phantom{= \;} + 2\sum_{1\leq i<j\leq m}(\kappa_i-\kappa_j)^2\,\overline{\mathrm{Sect}}(E_i\wedge E_j) + \left(\mathrm{trace}(A^2)\right)^2
	\end{align*}
	at $q$ and we can estimate
	\begin{align*}
		2\sum_{1\leq i<j\leq m}(\kappa_i-\kappa_j)^2\,\overline{\mathrm{Sect}}(E_i\wedge E_j) & \geq 2c_2\sum_{1\leq i<j\leq m}(\kappa_i-\kappa_j)^2 \\
		& = 2mc_2\sum_{i=1}^m \kappa_i^2 = 2mc_2\mathrm{trace}(A^2), \\
		\overline{\mathrm{Ric}}(N,N)\mathrm{trace}(A^2) & \geq c_1\mathrm{trace}(A^2).
	\end{align*}
	So, the function $u=\mathrm{trace}(A^2)$ satisfies $\Delta u \geq 2(c_1 + 2 m c_2 + u)u$. For any given unit vector $X\in T_q M$, by choosing an orthonormal basis $\{e_1,\dots,e_{m-1}\}$ of $X^{\bot}\subseteq T_q M$ we deduce from Gauss equations (see formulas (\ref{gauss}) and (\ref{ric}) in Section \ref{gausseq}) that
	$$
		\mathrm{Ric}(X,X) = \sum_{i=1}^{m-1}\overline{\mathrm{Sect}}(X\wedge e_i) + |AX|^2 \geq (m-1)c_2.
	$$
	By bilinearity, it follows that $\mathrm{Ric}(X,X)\geq(m-1)c_2|X|^2$ for every $X\in TM$. Since $M$ is complete, the Omori-Yau maximum principle for the Laplacian on $M$. Applying Theorem \ref{KO} with the choice $F(t)= t^2$, we deduce that $u$ is bounded above and that $u^{\ast} = \sup_M u$ satisfies
	$$
		(c_1 + 2mc_2 + u^{\ast}) u^{\ast} \leq 0,
	$$
	that is, $u^{\ast} \leq \max\left\{0,-(c_1+2mc_2)\right\}$.
\end{demo}

\begin{rem}
	We remark that in Theorem B of \cite{N} the general estimate (\ref{Abound}) is not stated and it is only proved that $\psi$ is totally geodesic when $c_1+2mc_2\geq0$.
\end{rem}

\begin{teor}
	\label{teocpt}
	Let $\overline{M}$ be a spacetime such that $\overline{\mathrm{Ric}}(Z,Z) \geq 0$ for every timelike vector $Z\in T\overline{M}$. If $\psi: M \longrightarrow \overline{M}$ is a maximal hypersurface, then $\psi$ is stable. If $M$ is also compact, then $\psi$ is totally geodesic.
\end{teor}

\begin{demo}
	The unit normal vector $N$ on $M$ is timelike so $\overline{\mathrm{Ric}}(N,N)\geq0$ and we have
	$$
		\int_M |\nabla\phi|^2 + \left( \overline{\mathrm{Ric}}(N,N)+\mathrm{trace}(A^2) \right)\phi^2 \geq 0 \qquad \text{for every } \phi \in C_c^\infty (M).
	$$
	Therefore $\psi$ is stable by definition and from Lemma \ref{lemeq} we deduce the existence of a positive function $u$ satisfying $\Delta u - \left( \overline{\mathrm{Ric}}(N,N) + \mathrm{trace}(A^2) \right) u =0$. If $M$ is compact we have
	$$
		\int_M \left( \overline{\mathrm{Ric}}(N,N)+\mathrm{trace}(A^2) \right)u = \int_M \Delta u = 0
	$$
	by the divergence theorem. But $\overline{\mathrm{Ric}}(N,N) \geq 0$, $\mathrm{trace}(A^2) \geq 0$ and $u>0$, so $\mathrm{trace}(A^2) \equiv 0$, that is, $\psi$ is totally geodesic.
\end{demo}

From Theorems \ref{teoni2} and \ref{teocpt} we easily deduce the next

\begin{teor}
	\label{teocsc}
	Let $\overline{M}$ be a spacetime of dimension $m+1$ and constant curvature $\overline{\kappa}$ and let $\psi : M \longrightarrow \overline{M}$ be a complete maximal hypersurface.
	\begin{itemize}[noitemsep,nolistsep]
		\item [i)] If $\overline{\kappa}>0$ then $M$ is compact and the immersion $\psi$ is totally geodesic and unstable.
		\item [ii)] If $\overline{\kappa}=0$ then $\psi$ is totally geodesic and stable.
		\item [iii)] If $\overline{\kappa}<0$ then $\psi$ is stable and the shape operator $A$ and the scalar curvature $\mathrm{S}$ of $M$ satisfy
		\begin{equation}
			\label{Abd}
			\mathrm{trace}(A^2)\leq -m\overline{\kappa}, \qquad \mathrm{S} \leq (m-2)m\overline{\kappa}.
		\end{equation}
		If $M$ is also compact, then $\psi$ is totally geodesic.
	\end{itemize}
\end{teor}

\begin{demo}
	We have $\overline{\mathrm{Ric}}(Z,Z) = -m\overline{\kappa}$ for all unit timelike vectors $Z\in T\overline{M}$, so conditions (\ref{ricbound1}) and (\ref{sectbound1}) are satisfied with $c_1=-m\overline{\kappa}$, $c_2=\overline{\kappa}$. Hence, $c_1 + 2mc_2=m\overline{\kappa}$ and by Theorem \ref{teoni2} we have $\mathrm{trace}(A^2)\leq\max\{0,-m\overline{\kappa}\}$. If $\overline{\kappa}\geq0$ it follows that $\psi$ is totally geodesic, while for $\overline{\kappa}<0$ we obtain $\mathrm{trace}(A^2) \leq -m\overline{\kappa}$, that by (\ref{scalmax}) is equivalent to saying that the scalar curvature $\mathrm{S}$ of $M$ satisfies
	$$ \mathrm{S} = (m-1)m\overline{\kappa} + \mathrm{trace}(A^2) \leq (m-2)m\overline{\kappa}. $$
	If $\overline{\kappa}>0$ then $M$ has constant positive sectional curvature $\overline{\kappa}$ and therefore it must be compact by the Bonnet-Myers theorem. Since $\psi$ is totally geodesic, $\overline{\mathrm{Ric}}(N,N)+\mathrm{trace}(A^2)=-m\overline{\kappa}<0$ on $M$ and the constant, compactly supported function $\phi\equiv1$ verifies
	$$ \frac{\int_{M}|\nabla\phi|^2 + \left( \overline{\mathrm{Ric}}(N,N)+\mathrm{trace}(A^2) \right) \phi^2}{\int_M \phi^2} = -m\overline{\kappa} < 0, $$
	implying $\lambda_1^L (M)<0$ by (\ref{lambdastab}). When $\overline{\kappa} \leq 0$ we have $\overline{\mathrm{Ric}}(Z,Z) = -m\overline{\kappa} \geq 0$ for every unit timelike vector $Z\in T\overline{M}$, so the other statements are direct consequences of Theorem \ref{teocpt}.
\end{demo}

The following Theorem \ref{teosup1} is a refinement of Theorem \ref{teocsc} for maximal surfaces in $3$-dimensional spacetimes of constant sectional curvature. Let us recall from (\ref{scalmax}) that the Gaussian curvature $\mathrm{K}$ of such a surface satisfies
\begin{equation}
	\label{sca}
	\tcr{2\mathrm{K}} = \mathrm{S} = \overline{\mathrm{S}} + 2 \ \overline{\mathrm{Ric}}(N, N) + \mathrm{trace}(A^2),
\end{equation}
\noindent where $\overline{\mathrm{S}}$ denotes the scalar curvature of the ambient spacetime $\overline{M}$.

\begin{teor}
	\label{teosup1}
	Let $\psi:M\to\overline{M}$ be a complete maximal \tcred{oriented} surface in a $3$-dimensional spacetime $\overline{M}$ of constant sectional curvature $\overline{\kappa}$.
	\begin{itemize}[noitemsep,nolistsep]
		\item [i)] If $\overline{\kappa}>0$ then $M$ is a totally geodesic, unstable round sphere of constant curvature $\overline{\kappa}$.
		\item [ii)] If $\overline{\kappa}=0$ then $M$ is totally geodesic, stable and it is either a Euclidean plane, a flat cylinder, or a flat torus.
		\item [iii)] If $\overline{\kappa}<0$ then $M$ is stable and has non-positive Gaussian curvature. If $M$ is compact then it is totally geodesic and its Euler characteristic satisfies
		\begin{equation}
			\label{eulchar}
			\chi(M) = \frac{\overline{\kappa}}{2\pi}\,\mathrm{Vol}(M).
		\end{equation}
		If $M$ is non-compact but its total curvature and its Euler characteristic are finite, then
		\begin{equation}
			\label{eulcharbis}
			\chi(M) \geq \frac{\overline{\kappa}}{2\pi}\,\mathrm{Vol}(M).
		\end{equation}
	\end{itemize}
\end{teor}

\begin{demo}
	By formula (\ref{sca}), the Gaussian curvature $\mathrm{K}$ of $M$ always satisfies
	$$
		\mathrm{K} = \frac{1}{2}\,\overline{\mathrm{S}} + \overline{\mathrm{Ric}}(N,N) + \frac{1}{2}\,\mathrm{trace}(A^2) = \overline{\kappa} + \frac{1}{2}\,\mathrm{trace}(A^2) \geq \overline{\kappa}.
	$$
	If $\overline{\kappa} > 0$ then, in view of Theorem \ref{teocsc}, we only have to show that $M$ is a topological sphere. Since $M$ is a compact surface of constant Gaussian curvature $\mathrm{K}=\overline{\kappa}$, by the Gauss-Bonnet theorem the Euler characteristic $\chi(M)$ satisfies
	$$ \chi(M) = \frac{1}{2\pi}\int_M \mathrm{K} = \frac{\overline{\kappa}}{2\pi}\,\mathrm{Vol}(M) > 0. $$
	Since $\chi(M)=2-2g$ with $g$ the topological genus of $M$, we conclude that $g=0$ and $M$ is a topological sphere.
	
	If $\overline{\kappa} = 0$ then $\psi$ is totally geodesic by Theorem \ref{teocsc} and $M$ is a flat surface. Note that all of the three cases described in point $ii)$ of the statement of the theorem can occur, for example when $M$ is a spacelike slice of a Lorentzian product $\overline{M} = \mathbb{R} \times F$ with metric $\overline{g} = -dt^2 + g_F$ and $(F,g_F)$ is a flat surface of one of the three above types.
	
	If $\overline{\kappa}<0$ and $M$ is compact, we obtain (\ref{eulchar}) by applying again the Gauss-Bonnet theorem. If $M$ is non-compact but has finite total curvature and finite Euler characteristic, inequality (\ref{eulcharbis}) follows by using the Cohn-Vossen's inequality (see page 86 in \cite{Os})
	$$ \chi(M) \geq \frac{1}{2\pi}\int_M \mathrm{K}, $$
	which is valid under our assumptions.
\end{demo}

We conclude this section by restating and proving Theorem \ref{intro:teosup2} from the Introduction under slightly more general hypotheses, see Remark \ref{ricnonneg} below.

\begin{teor}
	\label{teoindex1}
	Let $\overline{M}$ be a $3$-dimensional spacetime satisfying $\overline{\mathrm{S}} + \overline{\mathrm{Ric}}(Z,Z) \geq 0$ for every unit timelike vector $Z\in T\overline{M}$ and let $\psi: M \longrightarrow \overline{M}$ be a complete stable \tcred{oriented} maximal surface. If $\psi$ is stable at infinity and $\mathrm{K}_+ \in L^1(M)$, then also $\mathrm{K}_{-}\in L^1(M)$, where $\mathrm{K}_+$ and $\mathrm{K}_-$ are the positive and negative parts of the Gaussian curvature $\mathrm{K}$ of $M$.
\end{teor}

\noindent To prove Theorem \ref{teoindex1} we need the following result due to Fischer-Colbrie, see Theorem 1 of \cite{FC}.

\begin{lema}
	\label{confsup}
	Let $(M,g)$ be a complete Riemann surface with Gaussian curvature $\mathrm{K}$. If $u\in C^{\infty}(M)$ is a positive function such that $\Delta u\leq\mathrm{K}u$ on $M\setminus\overline{\Omega}$ for some relatively compact open set $\Omega\subseteq M$, then $(M,u^2g)$ is complete.
\end{lema}

\begin{proof}[Proof of Theorem \ref{teoindex1}]
	If $M$ is compact, then it has finite total curvature and we are done. Hence, suppose that $M$ is complete and non-compact. By (\ref{sca}) and since $\overline{\mathrm{S}} + \overline{\mathrm{Ric}}(N,N) \geq 0$, the Gaussian curvature of $M$ satisfies $2\mathrm{K} \geq \overline{\mathrm{Ric}}(N,N) + \mathrm{trace}(A^2)$. Let $\Omega_0\subseteq M$ be a relatively compact open set such that the stability operator $L = \Delta - \left( \overline{\mathrm{Ric}}(N,N) + \mathrm{trace}(A^2) \right)$ satisfies $\lambda_1^L(M\setminus\overline{\Omega_0})\geq0$. By the variational characterization (\ref{lambdastab}), for every $\phi\in C^{\infty}(M\setminus\overline{\Omega_0})$ we have
	\begin{align*}
		\int_{M\setminus\overline{\Omega_0}} |\nabla\phi|^2 + \mathrm{K}\phi^2 & \geq \frac{1}{2}\int_{M\setminus\overline{\Omega_0}} |\nabla\phi|^2 + 2\mathrm{K}\phi^2 \\
		& \geq \frac{1}{2}\int_{M\setminus\overline{\Omega_0}} |\nabla\phi|^2 + \left( \overline{\mathrm{Ric}}(N,N) + \mathrm{trace}(A^2) \right)\phi^2 \geq 0
	\end{align*}
	so the operator $\Delta-\mathrm{K}$ satisfies $\lambda_1^{\Delta-\mathrm{K}}(M\setminus\overline{\Omega_0})\geq0$	and by Lemma \ref{lemeq} there exists a positive solution $u_0$ of $\Delta u_0=\mathrm{K}u_0$ on $M\setminus\overline{\Omega_0}$. By standard elliptic regularity results, $u_0$ is smooth. Let $\Omega\subseteq M$ be a relatively compact open set such that $\overline{\Omega_0}\subseteq \Omega$ and let $u\in C^{\infty}(M)$ be a positive function such that $u=u_0$ on $M\setminus\overline{\Omega}$, so that
	\begin{equation}
		\label{Deltau}
		\Delta u = \mathrm{K} u \qquad \text{on } M\setminus\overline{\Omega}.
	\end{equation}
	By Lemma \ref{confsup}, $M$ is complete in the conformally deformed metric $\tilde{g} = u^2g$ and the Gaussian curvature $\tilde{\mathrm{K}}$ of $(M,\tilde{g})$ is nonnegative on $M\setminus \Omega$ since
	\begin{equation}
		\tilde{\mathrm{K}} = \frac{1}{u^2}(\mathrm{K} - \Delta\log u) = \frac{\mathrm{K}u - \Delta u}{u^3} + \frac{|\nabla u|^2}{u^4}.
	\end{equation}
	
	Let $o\in M$ be a given point and let $B_r$ denote the geodesic ball of $(M,\tilde{g})$ centered at $o$ with radius $r$. Completeness of $(M,\tilde{g})$ by the Hopf-Rinow theorem enables us to choose $R>0$ sufficiently large so that $\Omega\subseteq B_R$. As $\tilde{K}\geq0$ on $M\setminus B_R$, by the volume comparison theorem there exists $C>0$ such that $\mathrm{Vol}(\partial B_r) \leq Cr$ for a.e. $r\geq R$, with $\mathrm{Vol}(\partial B_r)$ the $1$-dimensional Hausdorff measure of $\partial B_r$ induced by $\tilde{g}$. Note that this is well defined for a.e. $r>0$. Thus
	\begin{equation}
		\label{igc15}
		\frac{1}{\mathrm{Vol}(\partial B_r)} \notin L^1(+ \infty).
	\end{equation}
	We let $w(r) = \mathrm{Vol}(\partial B_r)$ and we set 
	\begin{equation}
		\label{igc17}
		A(r) = \frac{1}{w(r)} \int_{\partial B_r} - \frac{\mathrm{K}}{u^2} \, d\mathrm{Vol}_{\tilde{g}}.
	\end{equation}
	We then consider the Cauchy problem
	\begin{equation}
		\label{CP}
		\begin{cases}
			(w(r)z')' + A(r)w(r)z = 0 & \text{on } \mathbb R^+ \\
			z(0^+) = 1, \quad (wz')(0^+) = 0.
		\end{cases}
	\end{equation}
	Applying Proposition 4.2 of \cite{BMR} we deduce the existence of a weak solution $z$ of (\ref{CP}). We now reason by contradiction and we suppose that $\mathrm{K}_+ \in L^1(M,d\mathrm{Vol}_g)$ while $\mathrm{K}_- \notin L^1(M,d\mathrm{Vol}_g)$. Then
	\begin{equation}
		\lim_{r\to+\infty} \int_{B_r} \mathrm{K} \, d\mathrm{Vol}_g = -\infty.
	\end{equation}
	By the coarea formula,
	\begin{equation}
		\label{igc18}
		\int_0^r A(s) w(s) ds = \int_{B_r} - \frac{\mathrm{K}}{u^2} \, d\mathrm{Vol}_{\tilde{g}} = \int_{B_r} - \mathrm{K} \, d\mathrm{Vol}_g
	\end{equation}
	and therefore
	\begin{equation}
		\label{nonL1}
		\lim_{r\to+\infty} \int_0^r A(s)w(s) ds = +\infty.
	\end{equation}
	This, together with (\ref{igc15}) implies, by Corollary 2.9 of \cite{MMR}, that $z$ is oscillatory. Let $R\leq R_1<R_2$ be two consecutive zeros of $z$ such that $z>0$ on $(R_1,R_2)$. Define a function $\varphi\in\mathrm{Lip}_0 (\overline{B_{R_2}}\setminus B_{R_1})$ by setting $\varphi(x) := z(r(x))$ for each $x\in B_{R_2}\setminus \overline{B_{R_1}}$, with $r(x)$ the distance from $x$ to $o$ in the metric $\tilde{g}$. By the coarea formula and (\ref{igc17})
	\begin{equation}
		\label{igc19}
		\begin{split}
		\int_{B_{R_2}\setminus \overline{B_{R_1}}}|\tilde{\nabla}\varphi|_{\tilde{g}}^2 \, d\mathrm{Vol}_{\tilde{g}} & + \int_{B_{R_2}\setminus \overline{B_{R_1}}} \frac{\mathrm{K}}{u^2}\varphi^2 \, d\mathrm{Vol}_{\tilde{g}} = \\
		& = \int_{R_1}^{R_2} w(s)z'(s)^2 ds - \int_{R_1}^{R_2} A(s)w(s)z(s)^2 ds.
		\end{split}
	\end{equation}
	Since $z$ is a weak solution of (\ref{CP}), using $z\,\chi_{[R_1,R_2]} \in \mathrm{Lip}_0([R_1,R_2])$ as a test function we get
	\begin{equation}
		\label{igc20}
		\tcr{\int_{R_1}^{R_2} w(s)z'(s)^2 ds = \int_{R_1}^{R_2} A(s)w(s)z(s)^2 ds.}
	\end{equation}
	\tcr{Collecting (\ref{igc19}) and (\ref{igc20}), by the monotonicity property of eigenvalues we obtain}
	$$
		\tcr{\lambda_1^{\tilde{L}} (M \setminus \overline{\Omega}) < \lambda_1^{\tilde{L}} (B_{R_2}\setminus \overline{B_{R_1}}) \leq \frac{\int_{B_{R_2}\setminus \overline{B_{R_1}}} |\tilde{\nabla}\varphi|_{\tilde{g}}^2 \, d\mathrm{Vol}_{\tilde{g}} + \int_{B_{R_2}\setminus \overline{B_{R_1}}} \frac{\mathrm{K}}{u^2} \varphi^2 \, d\mathrm{Vol}_{\tilde{g}} }{\int_{B_{R_2}\setminus \overline{B_{R_1}}}\varphi^2} = 0,}
	$$
	\tcr{where $\tilde{L}$ is the operator defined by}
	$$
		\tcr{\tilde{L} = \tilde{\Delta} - \frac{\mathrm{K}}{u^2}},
	$$
	with $\tilde{\Delta}$ the Laplace-Beltrami operator of $(M,\tilde{g})$. Since $u^2 \tilde{\Delta} f = \Delta f$ on $M$ for every $f\in C^2(M)$, we see from (\ref{Deltau}) that $u$ is a positive solution of $\tilde{L}u = 0$ on $M\setminus\overline{\Omega}$, so $\lambda_1^{\tilde{L}} (M \setminus \overline{\Omega})\geq0$ by Lemma \ref{lemeq} and we have reached the desired contradiction.
\end{proof}

\begin{rem}
	\label{ricnonneg}
	If $\{E_1,E_2,Z\}$ is a local Lorentz orthonormal frame on $T\overline{M}$ and $Z$ is timelike then $\overline{\mathrm{S}}+\overline{\mathrm{Ric}}(Z,Z) = \overline{\mathrm{Ric}}(E_1,E_1) + \overline{\mathrm{Ric}}(E_2,E_2)$, so Theorem \ref{intro:teosup2} is indeed a consequence of Theorem \ref{teoindex1}.
\end{rem}

\section{Maximal hypersurfaces in GRW spacetimes} \label{segrw}

\tcr{Let $(F,g_{F})$ be an $m$-dimensional (connected)
Riemannian manifold, $m\geq2$, $I$ an open interval in $\mathbb{R}$ endowed with the metric $-dt^2$ and $\rho$ a
positive smooth function defined on $I$. The Generalized Robertson-Walker (GRW) spacetime $\overline{M}=I\times_{\rho} F$, with fiber $(F,g_F)$, base $(I,-dt^2)$ and warping function $\rho$, is the product manifold $\overline{M}=I \times F$ endowed with the Lorentzian metric}
\begin{equation}
	\label{metrica}
	\tcr{\overline{g} = -\pi^*_{I} (dt^2) +\rho(\pi_I)^2 \, \pi_{F}^* (g_{F}),}
\end{equation}
where, respectively, $\pi_{I}$ and $\pi_{F}$ denote the projections from $\overline{M}$ onto $I$ and $F$. If the fiber has constant sectional curvature, $\overline{M}$ is simply called a Robertson-Walker spacetime.

In any GRW spacetime $\overline{M}=I\times_\rho F$, the
coordinate vector field $\partial_t:=\partial/\partial t$ is
a unit timelike vector field and hence $\overline{M}$ is time-orientable. With a slight abuse of notation, we write $\rho(t)$, $\rho'(t)$, $\rho''(t)$ to denote $\rho\circ\pi_I$, $\rho'\circ\pi_I$, $\rho''\circ\pi_I$. If we consider the timelike vector field
$$
	\tcr{T}: = \rho(\tcr{t})\,\partial_t,
$$
from the relation between the Levi-Civita connection of $\overline{M}$ and those of the base and
the fiber (see Corollary 7.35 of \cite{O'N}) it follows that
\begin{equation}
	\label{conexion} 
	\overline{\nabla}_X \tcr{T} = \rho'(\tcr{t}) X, 
\end{equation}
for any $X\in \mathfrak{X}(\overline{M})$, where $\overline{\nabla}$
is the Levi-Civita connection of the Lorentzian metric (\ref{metrica}). Thus, $\tcr{T}$ is conformal and its metrically equivalent $1$-form is closed, that is, $\tcr{T}$ is a closed conformal vector field. The curvature tensors of $\overline{M}$ are given by the following formulas.

\begin{lema}
	The GRW spacetime $\overline{M} = I\times_{\rho} F$ has Riemann and Ricci curvature tensors given by
	\begin{align}
		\label{riemGRW}
		\overline{\mathrm{Riem}} & = \rho(t)^2{\pi_F}^{\ast}\mathrm{Riem}^F + \rho(t)^2\left(\rho'(t)^2 - \rho(t)\rho''(t)\right){\pi_F}^{\ast}(g_F\KN g_F) + \frac{\rho''(t)}{\rho(t)}(\overline{g}\KN\overline{g}), \\
		\label{ricGRW}
		\overline{\mathrm{Ric}} & = {\pi_F}^{\ast}\mathrm{Ric}^F - (m-1)\left(\rho(t)\rho''(t) - \rho'(t)^2\right){\pi_F}^{\ast}g_F + m\frac{\rho''(t)}{\rho(t)}\overline{g},
	\end{align}
	where $\mathrm{Riem}^F$ and $\mathrm{Ric}^F$ are the Riemann and Ricci tensors of $(F,g_F)$ and $\KN$ denotes Kulkarni-Nomizu product.
\end{lema}

\begin{proof}
	The gradient and the Hessian of the warping function $\rho$ in the base $(I,-dt^2)$ are given by $\nabla^I\rho = -\rho'\partial_t$ and $\mathrm{Hess}(\rho) = \rho''dt\otimes dt$. Since $\overline{\nabla}_{\partial_t}\partial_t = 0$ on $\overline{M}$, the lift $-\rho'(t)\partial_t \in \mathfrak X(\overline{M})$ of $\nabla^I\rho$ satisfies $\overline{\nabla}_{\partial_t}(-\rho'(t)\partial_t) = -\rho''(t)\partial_t$ on $\overline{M}$. Let $p=(t,x)\in\overline{M}$ be given and let $X,Y,U,V,W\in T_p\overline{M}$ be such that
	$$
	{\pi_F}_{\ast}(X) = {\pi_F}_{\ast}(Y) = 0 \in T_x F, \qquad {\pi_I}_{\ast}(U) = {\pi_I}_{\ast}(V) = {\pi_I}_{\ast}(W) = 0 \in T_t I.
	$$
	From formulas (2)-(5) of Proposition 7.42 of \cite{O'N} we get
	\begin{align*}
		\overline{\mathrm{Riem}}(V,W,U,\;\cdot\;) & = \rho(t)^2\pi_F^{\ast}\mathrm{Riem}^F(V,W,U,\;\cdot\;) \\
		& \phantom{=\;} + \frac{\rho'(t)^2}{\rho(t)^2}\left(\overline{g}(V,U)\overline{g}(W,\;\cdot\;) - \overline{g}(W,U)\overline{g}(V,\;\cdot\;)\right), \\
		\overline{\mathrm{Riem}}(V,X,W,\;\cdot\;) & = \frac{\rho''(t)}{\rho(t)}\overline{g}(V,W)\overline{g}(X,\;\cdot\;), \\
		\overline{\mathrm{Riem}}(V,X,Y,\;\cdot\;) & = -\frac{\rho''(t)}{\rho(t)}\overline{g}(X,Y)\overline{g}(V,\;\cdot\;)
	\end{align*}
	and from formulas (1)-(3) of Corollary 7.43 of \cite{O'N} we also have
	\begin{align*}
		\overline{\mathrm{Ric}}(V,\;\cdot\;) & = \pi_F^{\ast}\mathrm{Ric}^F(V,\;\cdot\;) + \left( \frac{\rho''(t)}{\rho(t)} + (m-1)\frac{\rho'(t)^2}{\rho(t)^2}\right)\overline{g}(V,\;\cdot\;), \\
		\overline{\mathrm{Ric}}(X,\;\cdot\;) & = m\frac{\rho''(t)}{\rho(t)}\overline{g}(X,\;\cdot\;)
	\end{align*}
	A direct computation shows that the RHS's of (\ref{riemGRW}) and (\ref{ricGRW}) also satisfy the identities above. By the symmetry properties of $\overline{\mathrm{Riem}}$, these identities uniquely determine its action on $T_p\overline{M}$.
\end{proof}

As we see from (\ref{riemGRW}), $\overline{M}$ has constant curvature $\overline{\kappa}$ if and only if the fiber $F$ has constant curvature $\kappa_F$ and the warping function $\rho$ satisfies
\begin{equation}
	\label{ccGRW}
	\overline{\kappa} = \frac{\rho''}{\rho} \qquad \text{and} \qquad \kappa_F = \rho\rho'' - (\rho')^2 \qquad \text{on } I.
\end{equation}
These equations are not independent. In fact, if there exists $C\in\mathbb{R}$ such that $\rho\rho'' - (\rho')^2 = C$ on an interval $I_0\subseteq I$, then
$$
	\left(\frac{\rho''}{\rho}\right)' = \left(\frac{C}{\rho^2}+\left(\frac{\rho'}{\rho}\right)^2\right)' = -\frac{2C\rho'}{\rho^3} + 2\frac{\rho'}{\rho}\left(\frac{\rho'}{\rho}\right)' = -\frac{2C\rho'}{\rho^3} + \frac{2\rho'}{\rho}\frac{\rho\rho''-(\rho')^2}{\rho^2} = 0,
$$
that is, $\rho''/\rho$ is constant on $I_0$. We also characterize GRW spacetimes with spacelike sectional curvatures bounded below.

\begin{lema}
	\label{charsectbound}
	Let $\overline{M}=I\times_{\rho}F$ be a GRW spacetime. For every $C_1\in\mathbb{R}$, the following are equivalent:
	\begin{itemize}[noitemsep,nolistsep]
		\item [i)] $\overline{\mathrm{Sect}}(\Pi) \geq C_1$ for every spacelike $2$-plane $\Pi\subseteq T\overline{M}$,
		\item [ii)] there exists $C_2\in\mathbb{R}$ such that
		\begin{alignat*}{2}
			\mathrm{Sect}^F(\Pi_0) & \geq C_2 && \qquad \text{for every $2$-plane } \Pi_0 \subseteq TF \\
			\frac{C_2 - \rho\rho'' + (\rho')^2}{\rho^2} & \geq \max\left\{C_1-\frac{\rho''}{\rho}, 0\right\} && \qquad \text{on } I.
		\end{alignat*}
	\end{itemize}
\end{lema}

\begin{demo}
	For every $t\in I$, define the $(0,4)$-tensor $V_t:TF\otimes TF\otimes TF\otimes TF\to C^{\infty}(F)$ by setting
	\begin{equation}
		\label{Vt}
		\begin{split}
		V_t(X,Y,Z,W) & = \rho(t)^2 \mathrm{Riem}^F(X,Y,Z,W) \\
		& \phantom{=\;} + \rho(t)^2\left(\rho'(t)^2 - \rho(t)\rho''(t)\right)(g_F\KN g_F)(X,Y,Z,W)
		\end{split}
	\end{equation}
	for every $X,Y,Z,W\in T_x F$, $x\in F$.
	
	Assume that i) holds. Let $(t,x)\in\overline{M}$, $\alpha\in\mathbb{R}$, $X,Y$ be given, with
	\begin{equation}
		\label{XY}
		X,Y\in T_x F \qquad \text{such that} \qquad g_F(X,X)=g_F(Y,Y)=1, g_F(X,Y)=0.
	\end{equation}
	The vectors
	\begin{equation}
		\label{E1E2}
		E_1 = \frac{1}{\rho(t)}\,X, \qquad E_2 = \frac{\cosh\alpha}{\rho(t)}\,Y + \sinh\alpha\,\partial_t
	\end{equation}
	belong to $T_{(t,x)}\overline{M}$ and satisfy $\overline{g}(E_1,E_1)=\overline{g}(E_2,E_2)=1$, $\overline{g}(E_1,E_2)=0$, so, by (\ref{riemGRW}), (\ref{Vt}) and (\ref{E1E2}),
	\begin{equation}
		\label{char1}
		C_1 \leq \overline{\mathrm{Sect}}(E_1\wedge E_2) = \overline{\mathrm{Riem}}(E_1,E_2,E_1,E_2) = \frac{\cosh^2\alpha}{\rho(t)^4}V_t(X,Y,X,Y) + \frac{\rho''(t)}{\rho(t)},
	\end{equation}
	that is,
	\begin{equation}
		\label{char2}
		\frac{\mathrm{Sect}^F(X\wedge Y)-\rho(t)\rho''(t)+\rho'(t)^2}{\rho(t)^2} = \frac{V_t(X,Y,X,Y)}{\rho(t)^4} \geq \frac{1}{\cosh^2\alpha}\left(C_1 - \frac{\rho''(t)}{\rho(t)}\right).
	\end{equation}
	For $\alpha=0$ and $\alpha\to+\infty$ we respectively get
	\begin{equation}
		\label{char3}
		\begin{split}
			\frac{\mathrm{Sect}^F(X\wedge Y)-\rho(t)\rho''(t)+\rho'(t)^2}{\rho(t)^2} & \geq C_1 - \frac{\rho''(t)}{\rho(t)}, \\
			\frac{\mathrm{Sect}^F(X\wedge Y)-\rho(t)\rho''(t)+\rho'(t)^2}{\rho(t)^2} & \geq 0.
		\end{split}
	\end{equation}
	For any fixed $t$, these inequalities must hold for every $x\in F$ and for every $X,Y$ as in (\ref{XY}), so $C=\inf\{ \mathrm{Sect}^F(\Pi_0) : \Pi_0\subseteq TF \text{ is a $2$-plane} \}$ is finite and ii) follows with $C_2=C$.
	
	Vice versa, assume that ii) holds. Let $(t,x)\in\overline{M}$ and $\Pi\subseteq T_{(t,x)}\overline{M}$ a spacelike $2$-plane be given. We can find a $\overline{g}$-orthonormal basis $\{E_1,E_2\}$ for $\Pi$ of the form (\ref{E1E2}), with $\alpha\in\mathbb{R}$ and $X,Y$ as in (\ref{XY}). Since inequalities (\ref{char3}) hold by assumption, we have (\ref{char2}) and therefore (\ref{char1}), that is, $\overline{\mathrm{Sect}}(\Pi)\geq C_1$. As $\Pi\subseteq T\overline{M}$ is arbitrarily given, we obtain i).
\end{demo}

Let $\psi:M\longrightarrow\overline{M}$ be a spacelike hypersurface immersed in the GRW spacetime $\overline{M}=I\times_{\rho}F$. Consider the unit timelike vector $N$ normal to $M$ with the same time orientation as $\partial_t$ and let $A$ and $H$ be the shape operator and the mean curvature of $\psi$ in the direction of $N$ as described in Section \ref{sepre}. The height function $\tau$ of the immersion $\psi$ onto the factor $I$ and the amplitude $\theta$ of the hyperbolic angle between $N$ and $\partial_t$ are given by
\begin{equation} \label{deftautheta}
	\tcr{\tau = \pi_I\circ\psi, \qquad \cosh\theta = - \overline{g}(N,\partial_t).}
\end{equation}
Note that $\theta$ is well defined (up to a sign) by the wrong-way Cauchy Schwarz inequality, since $N$ and $\partial_t$ are unit timelike vectors with the same time-orientation. As above, we write $\rho(\tau)$, $\rho'(\tau)$, $\rho''(\tau)$ to denote $\rho\circ\tau$, $\rho'\circ\tau$, $\rho''\circ\tau$. For a fixed $t_0\in I$, we set
\begin{equation} \label{defeta}
	\overline{\eta}(t) = \int_{t_0}^t \rho(s)ds \quad \forall t\in I, \qquad \eta = \overline{\eta}\circ\tau.
\end{equation}
Since $\rho>0$, $\overline{\eta}$ is strictly increasing on $I$. We also consider the positive function
\begin{equation}
	\label{defv}
	v=-\overline{g}(T,N)=\rho(\tau)\cosh\theta.
\end{equation}
In the sequel we will make extensive use of these auxiliary functions. Denoting by $\nabla$ and $\Delta$ the Levi-Civita connection and the Laplace-Beltrami operator of $(M,g)$, we have the following computational result.

\begin{lema}
	\label{comp1}
	\tcr{Let $\psi:M\longrightarrow\overline{M}$ be a spacelike hypersurface immersed in a GRW spacetime $\overline{M}=I\times_{\rho}F$ and let $T$, $N$, $A$, $H$, $\theta$, $\tau$, $\eta$ and $v$ be as above. Then}
	\begin{align}
		\label{normgradtau}
		\tcr{|\nabla\tau|^2} & \tcr{\, = \sinh^2\theta,} \\
		\label{normgradeta}
		\tcr{|\nabla\eta|^2} & \tcr{\, = \rho(\tau)^2\sinh^2\theta = v^2 - \rho(\tau)^2,} \\
		\label{normgradv}
		\tcr{|\nabla v|^2} & \tcr{\, \leq \mathrm{trace}(A^2)|\nabla\eta|^2 = \mathrm{trace}(A^2)(v^2-\rho(\tau)^2),} \\
		\label{lapeta}
		\tcr{\Delta\eta} & \tcr{\, = -m\rho'(\tau) + mHv,} \\
		\label{lapv}
		\tcr{\Delta v} & \tcr{\, = \left(\overline{\mathrm{Ric}}(N,N) + \mathrm{trace}(A^2) + m\frac{\rho''(\tau)}{\rho(\tau)}\right)v - mg(T^{\top},\nabla H) - mH\rho'(\tau),}
	\end{align}
	\tcr{where $T^{\top}$ is the tangential part of $T$ along $\psi$. If $\psi$ is maximal, then}
	\begin{align}
		\label{normgradvmax}
		\tcr{|\nabla v|^2} & \tcr{\, \leq \frac{m-1}{m}\mathrm{trace}(A^2)(v^2-\rho(\tau)^2),} \\
		\label{lapetamax}
		\tcr{\Delta\eta} & \tcr{\, = -m\rho'(\tau),} \\
		\label{lapvmax}
		\tcr{\Delta v} & \tcr{\, = \left(\overline{\mathrm{Ric}}(N,N) + \mathrm{trace}(A^2) + m\frac{\rho''(\tau)}{\rho(\tau)}\right)v.}
	\end{align}
\end{lema}

\begin{demo}
	\tcr{As $dt=\overline{g}(\;\cdot\;,\partial_t)$ in $\overline{M}$, we have that $\nabla\tau=-\partial_t^{\top}$ on $M$, where $\partial_t^{\top}$ is the tangential part of $\partial_t$ along $\psi$. From (\ref{defeta}) it follows that
	\begin{equation}
		\label{gradeta}
		\nabla\eta=\rho(\tau)\nabla\tau=-T^{\top}.
	\end{equation}
	Using the orthogonal decomposition $\partial_t = \partial_t^{\top} + \cosh\theta N$ and $\overline{g}(\partial_t,\partial_t)=\overline{g}(N,N)=-1$ we get 
	}
	$$
		\tcr{|\nabla\tau|^2 = g(\partial_t^{\top},\partial_t^{\top}) = \overline{g}(\partial_t,\partial_t) - \cosh^2\theta\,\overline{g}(N,N) = \sinh^2\theta, \qquad |\nabla\eta|^2 = \rho(\tau)^2\sinh^2\theta,}
	$$
	\tcr{proving (\ref{normgradtau}) and (\ref{normgradeta}) in view of (\ref{defv}).} Since the tangential component of $\tcr{T}$ along $\psi$ is $\tcr{T}^{\top} = \tcr{T} + \overline{g}(\tcr{T},N)N \tcr{\, = T - vN}$,
	a direct computation using (\ref{conexion}) gives
	\begin{equation}
		\label{gradv}
		\nabla v = A\tcr{T}^{\top} \tcr{= -A\nabla\eta.}
	\end{equation}
	\tcr{Denoting by $\lambda_1,\dots,\lambda_m$ the eigenvalues of $A:T_pM\to T_p M$ at a given point $p\in M$, we have}
	\begin{equation}
		\label{AX}
		\tcr{|AX|^2 = g(A^2X,X) \leq |X|^2 \max_{1\leq i\leq m}\lambda_i^2 \leq \mathrm{trace}(A^2) |X|^2}
	\end{equation}
	\tcr{for each vector $X\in T_p M$. So, (\ref{normgradv}) follows from (\ref{gradv}) and (\ref{normgradeta}). If $\psi$ is maximal, then $\mathrm{trace}(A)\equiv0$ and we obtain a ``refined Kato''-type inequality: for each $1\leq i\leq m$ we have $\lambda_i = -\sum_{1\leq j\leq m,j\neq i}\lambda_j$ and using Cauchy inequality we get}
	$$
		\tcr{\lambda_i^2 + \frac{1}{m-1}\lambda_i^2 = \lambda_i^2 + \frac{1}{m-1}\left(\sum_{\substack{1\leq j\leq m \\ j\neq i}}\lambda_j\right)^2 \leq \lambda_i^2 + \frac{m-1}{m-1}\sum_{\substack{1\leq j\leq m \\ j\neq i}}\lambda_j^2 = \mathrm{trace}(A^2).}
	$$
	\tcr{Hence, $\lambda_i^2\leq\frac{m-1}{m}\mathrm{trace}(A^2)$ for each $1\leq i\leq m$, thus $|AX|^2 \leq \frac{m-1}{m}\mathrm{trace}(A^2) |X|^2$ for every $X\in T_p M$, proving the refined version (\ref{normgradvmax}).}
	
	\tcr{In order to prove (\ref{lapeta}) and (\ref{lapv}), we recall that} Gauss and Weingarten formulas for the immersion $\psi$ are respectively given by
	\begin{equation}
		\label{GWF}
		\overline{\nabla}_X Y = \nabla_X Y - g(AX,Y)N \qquad \text{and} \qquad AX = -\overline{\nabla}_X N
	\end{equation}
	for any $X,Y\in\mathfrak X(M)$ and that the covariant derivative $\nabla A$ of $A$, defined by $(\nabla_X A)Y = \nabla_X(AY) - A(\nabla_X Y)$ for every $X,Y\in\mathfrak X(M)$, satisfies Codazzi equation
	\begin{equation}
		\label{coda}
		 g((\nabla_X A)Y,Z) = g((\nabla_Y A)X,Z) - \overline{g}(\overline{\mathrm{R}}(X,Y)N,Z)
	\end{equation}
	for every $X,Y,Z\in\mathfrak X(M)$. Taking the tangential component in (\ref{conexion}) and using (\ref{GWF}) and (\ref{coda}) together with the definition of $\nabla A$, we get
	\begin{align}
		\label{hesseta}
		\nabla_X \tcr{T}^{\top} & = -\rho(\tau)\,\overline{g}(N,\partial_t)AX + \rho'(\tau)X \tcr{ \, = vAX + \rho'(\tau)X}, \\
		\label{hessv}
		\nabla_X (AT^{\top}) & = (\nabla_X A)T^{\top} + A(\nabla_X T^{\top}) \\
		& = (\nabla_{T^{\top}}A)X - (\overline{\mathrm{R}}(X,T^{\top})N)^{\top} + vA^2 X + \rho'(\tau) AX
	\end{align}
	for any $X\in\mathfrak X(M)$. By definition, for every function $u\in C^{\infty}(M)$,
	$$
		\Delta u = \mathrm{trace}(\nabla(\nabla u)) = \sum_{i=1}^m g(\nabla_{E_i}\nabla u,E_i)
	$$
	for any choice of a local orthonormal frame $\{E_1,\dots,E_m\}$ on $TM$. By (\ref{gradeta}), (\ref{gradv}), (\ref{hesseta}), (\ref{hessv}) and since $\mathrm{trace}(A) = -mH$ and $\mathrm{trace}(\nabla_{T^{\top}}A) = \nabla_{T^{\top}}(\mathrm{trace}(A)) = -m\nabla_{T^{\top}}H$, we obtain
	$$
		\Delta \eta = -m\rho'(\tau) - mHv, \quad \Delta v = - \overline{\mathrm{Ric}}(T^{\top},N) + \mathrm{trace}(A^2)v - m\nabla_{T^{\top}}H - mH\rho'(\tau),
	$$
	and (\ref{lapeta}) is proved. Writing $\tcr{T}^{\top} = \tcr{T} - vN$, by (\ref{ricGRW}) we have
	\begin{equation}
		\label{r0}
		\tcr{-\overline{\mathrm{Ric}}(T^{\top},N) = -\overline{\mathrm{Ric}}(T,N) + \Ric(N,N)v = \left(m\frac{\rho''(\tau)}{\rho(\tau)} + \Ric(N,N)\right)v,}
	\end{equation}
	\tcr{as ${\pi_F}^{\ast}\mathrm{Ric}^F(\partial_t,\;\cdot\;)=0$ and $(dt\otimes dt)(T,N) = -\overline{g}(T,N) = \rho(\tau)\cosh\theta = v$. This concludes the proof of (\ref{lapv}). If $\psi$ is maximal, then $H\equiv0$ and (\ref{lapetamax}) and (\ref{lapvmax}) follow at once.}
\end{demo}

\begin{rem}
	The spacelike slices $\{t\}\times F$, $t\in I$ of $\overline{M}$ are totally umbilical hypersurfaces, in other words they satisfy $\mathrm{trace}(A^2) = mH^2$, and they have mean curvature $H=\rho'(t)/\rho(t)$ in the direction of the future-pointing normal. This is a consequence of (\ref{lapeta}) and (\ref{lapv}), as the image $\psi(M)$ of an immersed hypersurface $\psi:M\longrightarrow\overline{M}$ is contained in a spacelike slice if and only if $\tau$ is constant on $M$, in which case $N=\partial_t$ and $v\equiv\rho(t)$.
\end{rem}

As a first application of equation (\ref{lapetamax}), we prove Theorem \ref{intro:teoGRW1} of the Introduction as a corollary of the following result, which generalizes Theorem 3.7 of \cite{ARSc}.

\begin{teor}
	\label{teosu}
	Let $\psi: M \longrightarrow \overline{M}$ be a maximal hypersurface in a GRW spacetime $\overline{M} = I \times_{\rho} F$. If the weak maximum principle for the Laplacian holds on $M$ and $\psi(M)$ is contained in a slab $[a, b] \times F \subseteq I \times F$, then $\rho'(\tau^{\ast}) \geq0$ and $\rho'(\tau_{\ast}) \leq 0$, where $\tau^{\ast} = \sup_M \tau$, $\tau_{\ast} = \inf_M \tau$.
\end{teor}

\begin{demo}
	Let $\eta$ be defined as in (\ref{defeta}). Since $\tau(M)\subseteq[a,b]$, we have
	\begin{equation}
		\label{pri5}
		-\infty < \int_{t_0}^a \rho(s)ds = \overline{\eta}(a) \leq \eta \leq \overline{\eta}(b) = \int_{t_0}^b \rho(s) ds < +\infty \qquad \text{on } M.
	\end{equation}
	By the weak maximum principle applied to $\eta$ and $-\eta$, see Definition \ref{MP}, we can find two sequences of points $\{x_k\}_{k\in\mathbb{N}},\{y_k\}_{k\in\mathbb{N}}\subseteq M$ such that
	\begin{align*}
		\lim_{k\to+\infty} \eta(x_k) = \sup_M \eta, & \qquad \limsup_{k\to+\infty} \Delta \eta(x_k) \leq 0, \\
		\lim_{k\to+\infty} \eta(y_k) = \inf_{M} \eta, & \qquad \liminf_{k\to+\infty} \Delta \eta(y_k) \geq 0.
	\end{align*}
	We recall that $\eta$ is a strictly monotonic function of $\tau$ and that $\Delta\eta = -m\rho'(\tau)$ on $M$. Hence,
	$$
		-m\rho'(\tau^{\ast}) = \lim_{k\to+\infty} -m\rho'(\tau(x_k))) = \lim_{k\to+\infty} \Delta\eta(x_k) \leq 0
	$$
	and similarly we have $-m\rho'(\tau_{\ast}) \geq 0$.
\end{demo}

\begin{coro}
	\label{teoGRW1}
	Let $\overline{M} = I\times_{\rho}F$ be a GRW spacetime whose sectional curvatures on spacelike $2$-planes are bounded below and let $\psi:M\longrightarrow\overline{M}$ be a complete maximal hypersurface. If $\psi(M)$ is contained in a slab $[a,b]\times F$, then $\rho'(\tau^{\ast}) \geq 0$, $\rho'(\tau_{\ast})\leq 0$, where $\tau^{\ast} = \sup_M \tau$, $\tau_{\ast} = \inf_M \tau$. In particular, if there exist $t_0,t_1\in[a,b]$ such that $\rho'=0$ on $[t_0,t_1]$, then the following implications hold:
	\begin{itemize}[noitemsep,nolistsep]
		\item [i)] if $\rho'<0$ on $[a,t_0)$ and $\rho'>0$ on $(t_1,b]$, then $\psi(M)$ must intersect every spacelike slice $\{t\}\times F$ with $t_0\leq t\leq t_1$;
		\item [ii)] if $\rho'>0$ on $[a,t_0)$ and $\rho'<0$ on $(t_1,b]$, then $\psi(M) \subseteq [t_0,t_1]\times F$.
	\end{itemize}
\end{coro}

\begin{demo}
	As in the proof of Theorem \ref{teoni2}, since $M$ is maximal and there exists $C\in\mathbb{R}$ such that $\overline{\mathrm{Sect}}(\Pi) \geq C$ for every spacelike $2$-plane $\Pi\subseteq T\overline{M}$, the Ricci curvature of $M$ is bounded below by $(\dim M - 1)C$. Since $(M,g)$ is complete, the weak maximum principle for the Laplacian holds on $M$ and the main statement is a direct consequence of Theorem \ref{teosu}. If $t_0\in[a,b]$ is such that $\rho'<0$ (respectively, $\rho'>0$) on $[a,t_0)$, then $\tau^{\ast}\geq t_0$ (resp., $\tau_{\ast}\geq t_0$). Similarly, if $t_1\in[a,b]$ is such that $\rho'>0$ (resp., $\rho'<0$) on $(t_1,b]$, then $\tau_{\ast}\leq t_1$ (resp., $\tau^{\ast}\leq t_1$). This concludes the proof.
\end{demo}

The following consequence is a Frankel type result.

\begin{coro}
	\label{corodS}
	Let $\mathbb{S}^{m+1}_1(\overline{\kappa})$ be the $(m+1)$-dimensional de Sitter spacetime of constant curvature $\overline{\kappa}>0$ and let $\psi_0 : M_0 \longrightarrow \mathbb{S}^{m+1}_1(\overline{\kappa})$, $\psi_1 : M_1 \longrightarrow \mathbb{S}^{m+1}_1(\overline{\kappa})$ be two complete maximal hypersurfaces. Then $\psi_0(M_0) \cap \psi_1(M_1) \neq \varnothing$.
\end{coro}

\begin{demo}
	Let $F=\mathbb{S}^m(\overline{\kappa})$ be the standard $m$-sphere of constant curvature $\overline{\kappa}>0$, set $\rho(t) = \cosh(\sqrt{\overline{\kappa}}t)$ for every $t\in\mathbb{R}$ and let $\overline{M} = \mathbb{R}\times_{\rho}F$. The GRW spacetime $\overline{M}$ is isometric to $\mathbb{S}^{m+1}_1(\overline{\kappa})$ (see, for instance, page 339 of \cite{W}). By Theorem \ref{teocsc}, $M_0$ is compact and $\psi_0$ is totally geodesic, so there exists an isometry $\Psi : \mathbb{S}^{m+1}_1(\overline{\kappa}) \longrightarrow \overline{M}$ such that $\Psi\circ\psi_0 : M_0 \to \overline{M}$ sends $M_0$ into the totally geodesic spacelike slice $\{0\}\times F$. Since $M_0$ is compact and $F$ is connected, $(\Psi\circ\psi_0)(M_0) = \{0\}\times F$. By Theorem \ref{teocsc} again, $M_1$ is compact. Let $\psi := \Psi \circ \psi_1 : M_1 \to \overline{M}$. The projection of $\psi(M_1)$ on the $\mathbb{R}$-factor of $\overline{M}$ is compact. Moreover, $\rho'<0$ on $(-\infty,0)$ and $\rho'>0$ on $(0,+\infty)$. So, we apply point i) of Corollary \ref{teoGRW1} to obtain that $(\{0\}\times F) \cap \psi(M_1) \neq \varnothing$. Since $\psi_0(M_0) = \Psi^{-1}(\{0\}\times F)$, $\psi_1(M_1) = \Psi^{-1}(\psi(M_1))$, we obtain $\psi_0(M_0) \cap \psi_1(M_1) \neq \varnothing$.
\end{demo}

We are now ready to prove Theorem \ref{intro:teosup3} of the Introduction.

\begin{teor}
	\label{teosup3}
	Let $\psi:M\longrightarrow\overline{M}$ be a complete maximal surface in a $3$-dimensional GRW spacetime $\overline{M}=I\times_{\rho}F$. Suppose that $\overline{M}$ has nonnegative sectional curvatures on spacelike tangent $2$-planes. Then $\psi$ is totally geodesic and one of the following cases occurs:
	\begin{itemize}[noitemsep,nolistsep]
		\item [i)] $\psi(M)$ is a spacelike slice $\{t\}\times F$ for some $t\in I$ such that $\rho'(t)=0$,
		\item [ii)] $(F,g_F)$ is a flat, complete Riemann surface and $\overline{M}$ is the product manifold $\mathbb{R}\times F$ endowed with the flat metric $-dt^2+g_F$,
		\item [iii)] $(F,g_F)$ is a compact Riemann surface of constant positive Gaussian curvature, $M$ is a round sphere and the spacetime $\overline{M}$ has constant positive sectional curvature in the slab $\tau(M)\times F$. 
	\end{itemize}
\end{teor}
	
\begin{demo}
	As already remarked, the assumption that $\overline{M}$ has nonnegative sectional curvatures on spacelike $2$-planes implies that $\mathrm{Ric}\geq0$ on $M$ in the sense of quadratic forms, that is, the Gaussian curvature $\mathrm{K}$ of $M$ is nonnegative. Moreover, denoting by $\mathrm{K}^F$ the Gaussian curvature of $F$, by Lemma \ref{charsectbound} there exists $C_2\in\mathbb{R}$ such that, for each $(t,x)\in\overline{M}$,
	\begin{equation}
		\label{curvbound}
		\begin{split}
		\mathrm{K}^F(x) \geq C_2, \qquad \frac{C_2-\rho(t)\rho''(t)+\rho'(t)^2}{\rho(t)^2} & \geq 0, \\
		\frac{C_2-\rho(t)\rho''(t)+\rho'(t)^2}{\rho(t)^2} & \geq -\frac{\rho''(t)}{\rho(t)}.
		\end{split}
	\end{equation}
	
	Let $v=-\overline{g}(T,N)$ on $M$. By (\ref{lapvmax}), for any $\alpha\in\mathbb{R}$ the function $v^{-\alpha}$ satisfies
		\begin{equation}
		\label{lapvalpha1}
			\begin{split}
				\Delta v^{-\alpha} & = -\alpha v^{-\alpha-1}\Delta v + \alpha(\alpha+1) v^{-\alpha-2}|\nabla v|^2 \\
				& = -\alpha v^{-\alpha}\left(\overline{\mathrm{Ric}}(N,N) + \mathrm{trace}(A^2) + 2\frac{\rho''(\tau)}{\rho(\tau)}-(\alpha+1)\frac{|\nabla v|^2}{v^2}\right)
			\end{split}
		\end{equation}
	on $M$. Since $\pi_F^{\ast}g_F(N,N) = \rho(t)^{-2}(\overline{g} + dt\otimes dt)(N,N) = \rho(t)^{-2}(-1+\cosh^2\theta) = \rho(t)^{-2}\sinh^2\theta$,
	from (\ref{ricGRW}) and the first two inequalities in (\ref{curvbound}) we have
	\begin{equation}
		\label{curvbound2}
		\overline{\mathrm{Ric}}(N,N) + 2\frac{\rho''(t)}{\rho(t)} = \sinh^2\theta\left(\frac{\mathrm{K}^F(x)}{\rho(t)^2}-\frac{\rho''(t)}{\rho(t)}+\frac{\rho'(t)^2}{\rho(t)^2}\right) \geq 0
	\end{equation}
	and by (\ref{normgradvmax}) we deduce, for every $\alpha\in[-1,1]$,
	\begin{equation}
		\mathrm{trace}(A^2) - (\alpha+1)\frac{|\nabla v|^2}{v^2} \geq \frac{1-\alpha}{2}\mathrm{trace}(A^2) \geq 0.
	\end{equation}
	Inserting these inequalities into (\ref{lapvalpha1}) we find that for every $\alpha\in[0,1]$ the positive function $v^{-\alpha}$ is superharmonic on $M$. In particular, $1/v$ is a positive superharmonic function on $M$. If $(M,g)$ is complete, then it is parabolic because of its nonnegative Gaussian curvature, so $1/v$ is constant on $M$. Therefore $v$ is also constant and from (\ref{lapvmax}) we obtain
	\begin{equation} \label{vanishing}
		\overline{\mathrm{Ric}}(N,N) + 2\frac{\rho''(\tau)}{\rho(\tau)} \equiv 0, \qquad \mathrm{trace}(A^2) \equiv 0 \qquad \text{on } M
	\end{equation}
	in view of (\ref{curvbound2}). Hence, $\psi$ is totally geodesic.
	
	Suppose that $\psi(M)$ is not a spacelike slice: then $\tau$ is not constant on $M$ and by (\ref{normgradtau}) the hyperbolic angle $\theta$ is not identically null, so $M_0:=\{q\in M:\sinh^2\theta(q)\neq0\}$	is a nonempty open subset of $M$. Since $\overline{\eta}$ is a strictly increasing function on $I$, the function $\eta$ defined in (\ref{defeta}) is nonconstant. Moreover, equation (\ref{lapetamax}) reads
	\begin{equation}
		\label{lapetamax1}
		\Delta \eta = -2(\rho'\circ\overline{\eta}^{-1})(\eta)
	\end{equation}
	and the function $-2(\rho'\circ\overline{\eta}^{-1})$ is of class $C^1$ on its domain; hence the unique continuation property holds for equation (\ref{lapetamax1}), that is, $\eta$ is constant on some nonempty open subset of $M$ if and only if it is constant on $M$ (see Theorem A.5 of \cite{PRS2}). Therefore, $M_0$ is a dense open subset of $M$. Finally, $v=\rho(\tau)\cosh\theta\geq\rho(\tau)$ implies that $\rho$ is bounded on $\tau(M)\subseteq I$, as $v$ is constant. We set
	\begin{equation}
		\label{I0F0}
		I_0 = \{ t\in I : \rho(t)\rho''(t) - \rho'(t)^2 = C_2 \}, \qquad F_0 = \{ x\in F : \mathrm{K}^F(x) = C_2 \}.
	\end{equation}
	Let $q\in M_0$ be given and set $(t,x)=\psi(q)\in\overline{M}$. From (\ref{vanishing}) it follows that (\ref{curvbound2}) holds with the equality sign. Since $\sinh^2\theta(q)\neq0$, the same is true for the first two inequalities in (\ref{curvbound}). Therefore, $\tau(M_0)\subseteq I_0$ and $(\pi_F\circ\psi)(M_0)\subseteq F_0$. Note that $I_0$ and $F_0$ are closed in $I$ and $F$, respectively, and that $\tau$ is constant on each connected component of $M\setminus M_0$. Hence, $\tau(M) = \tau(\overline{M_0}) \subseteq I_0$, $(\pi_F\circ\psi)(M) = (\pi_F\circ\psi)(\overline{M_0}) \subseteq F_0$. As observed at the beginning of this section, $\rho''/\rho$ is constant on every interval contained in $I_0$. Since $M$ is connected, $\tau(M)$ is an interval and there exists $\overline{\kappa}\in\mathbb{R}$ such that
	\begin{equation}
		\label{rho''}
		\frac{\rho''(t)}{\rho(t)} = \overline{\kappa} \qquad \text{for each } t\in \tau(M).
	\end{equation}
	By the third inequality in (\ref{curvbound}), $\overline{\kappa}\geq0$. So far, we have proved that $\overline{M}$ has constant curvature $\overline{\kappa}\geq0$ in the cylinder
	\begin{equation}
		\label{Sigma}
		\Sigma = \tau(M) \times (\pi_F\circ\psi)(M) \subseteq I_0\times F_0 \subseteq I\times F.
	\end{equation}
	We conclude the proof by showing that $(\pi_F\circ\psi)(M)=F$ and that in case where $\overline{\kappa}=0$ it must be $\tau(M)=\mathbb{R}$ and $\rho'\equiv0$ on $\mathbb{R}$.
	
	First suppose that $\overline{\kappa}=0$. Then, $\rho'$ is constant on $\tau(M)$. If $\rho'\equiv C_1$ on $\tau(M)$ for a positive constant $C_1>0$, then $\eta$ is a nonconstant superharmonic function on $M$ because it satisfies equation (\ref{lapetamax1}), so it cannot be bounded above on the parabolic surface $M$. This implies that $\tau(M)$ is not bounded above, otherwise we would have
	\[
		\sup_M\eta \leq \int_{t_0}^{\sup\tau(M)} \rho(s)ds = (\sup\tau(M)-t_0)\,\rho(t_0) + \frac{C_1}{2}\,(\sup\tau(M)-t_0)^2 < +\infty.
	\]
	Since $\rho'\equiv C_1>0$ on $\tau(M)$ and $\tau(M)$ is not bounded above, we obtain that $\rho$ is not bounded on $\tau(M)$ and we reach a contradiction. Similarly, we conclude that $\rho'$ cannot be identically equal to a negative constant on $\tau(M)$, so we are left with the case where $\rho'\equiv0$ on $\tau(M)$. In this case, by (\ref{lapetamax1}) we have that $\eta$ is a nonconstant harmonic function on $M$ and therefore it is not bounded above nor below. Arguing as above we can show that $\tau(M)$ is not bounded above nor below, so $\mathbb{R}=\tau(M)\subseteq I\subseteq\mathbb{R}$ and we conclude that $I=\mathbb{R}$. By (\ref{I0F0}) and (\ref{Sigma}), since $\rho'\equiv0$ on $\tau(M)$ we deduce $C_2=0$ and $\mathrm{K}^F=0$ on $F_0$. If we set $F' = (\pi_F\circ\psi)(M)$ and we endow the surface $M$ with the metric $\sigma:=(\pi_F\circ\psi)^{\ast}g_F$, then $(M,\sigma)$ is complete because $\sigma = g+d\tau^2 \geq g$ and $(\pi_F\circ\psi):M\to F'$ is a local Riemannian isometry, so $(F',(g_F)|_{F'})$ is also complete and we conclude that $F=F'\subseteq F_0\subseteq F$. In particular, $F_0=F$ and $\mathrm{K}^F\equiv0$ on $F$.
	
	Now, suppose that (\ref{rho''}) holds with $\overline{\kappa}>0$. $\psi(M)$ is contained in the cylinder $\Sigma$ and $\psi$ is totally geodesic, so $M$ has constant positive Gaussian curvature $\overline{\kappa}$ and therefore it is compact by Bonnet theorem. In particular, from Theorem \ref{teosup1} it follows that $(M,g)$ is a round sphere. The map $\pi_F\circ\psi$ is continuous and open (being a local diffeomorphism), so $(\pi_F\circ\psi)(M)$ is compact and open in $F$. Since $F$ is connected, we conclude $(\pi_F\circ\psi)(M)=F$. Moreover, since the second of (\ref{curvbound}) holds with the equality sign for every $t\in\tau(M)$, we have $(\log\rho)''=C_2/\rho^2$ on $\tau(M)$. Suppose, by contradiction, that $C_2\leq0$. Then $\log\rho$ is concave and the sign of $\rho'$ is nonincreasing on $\tau(M)$. As $\tau(M)$ is compact, we can apply the last statement of Corollary \ref{teoGRW1} to deduce that $\rho'(\tau)\equiv0$ on $M$ and by (\ref{lapetamax1}) we get that $\eta$ is a nonconstant harmonic function on the compact surface $M$, contradiction. Therefore, $C_2>0$ and by (\ref{I0F0}) we conclude that $F=F_0$ has constant positive Gaussian curvature.
\end{demo}

We conclude this section with the following two results.

\begin{teor}
	\label{teodecc}
	Let $\overline{M} = I \times_{\rho} F$ be a GRW spacetime with $\rho'' \leq 0$. Let $\psi : M \longrightarrow \overline{M}$ be a complete, non-compact, maximal hypersurface in $\overline{M}$ such that, for some $o\in M$ and for some (hence, any) $a>0$,
	\begin{equation}
		\label{vl1}
		\lim_{r\to+\infty} \int_a^r \frac{dt}{\mathrm{Vol}(\partial B_t)} = +\infty,
	\end{equation}
	where $B_t$ is the geodesic ball of $(M,g)$ centered at $o$ with radius $t$. Then
	\begin{equation}
		\label{liinfr}
		\limsup_{r\to+\infty} \int_{B_r} \left( \overline{\mathrm{Ric}}(N, N) + \mathrm{trace}(A^2) \right) \geq 0.
	\end{equation}
\end{teor}

\begin{demo}
	From (\ref{lapvmax}) we know that the positive function $v = -\overline{g}(\tcr{T}, N)$ satisfies
	\begin{equation}
		\label{lav1}
		Lv = \Delta v - \left( \overline{\mathrm{Ric}}(N, N) + \mathrm{trace}(A^2) \right) v = m \frac{\rho''(\tau)}{\rho(\tau)} v \leq 0.
	\end{equation}
	Therefore, by Lemma \ref{lemeq}, the operator $L$ satisfies $\lambda_1^L(M)\geq0$. Set
	\begin{equation}
		\label{lav2}
		w(r)=\mathrm{Vol}(\partial B_r), \qquad A(r) = - \frac{1}{w(r)} \int_{\partial B_r} \left( \overline{\mathrm{Ric}}(N,N) + \mathrm{trace}(A^2) \right)
	\end{equation}
	and consider the weak solution $z \in \mathrm{Lip}_{\mathrm{loc}}(\mathbb{R}_0^+)$ of the Cauchy problem
	\begin{equation}
		\label{CP2}
		\begin{cases}
			(w(r)z')' + A(r)w(r)z = 0 & \text{on } \mathbb R^+ \\
			z(0^+) = 1, \quad (wz')(0^+) = 0.
		\end{cases}
	\end{equation}
	We have $z>0$ on $\mathbb{R}^+_0$. If not, let $R_0$ be the first positive zero of $z$ and set $\varphi(x) = z(r(x))$ for $x\in B_{R_0}$, with $r(x)$ the distance from $x$ to $o$ in $(M,g)$. Since $z$ solves problem (\ref{CP2}), by Proposition \ref{propo2} and using the coarea formula as in the proof of Theorem \ref{teoindex1} we get
	$$
		0 \leq \lambda_1^L(M) < \lambda_1^L(B_{R_0}) \leq \frac{\int_0^{R_0} w(r)z'(r)^2 dr + \int_0^{R_0} A(r)w(r)z(r)^2 dr}{\int_0^{R_0} w(r)z(r)^2} = 0,
	$$
	contradiction. So, problem (\ref{CP2}) has a positive solution and (\ref{vl1}) holds. By Theorem 2.8 of \cite{MMR},
	\begin{equation}
		\label{lav3}
		0 \geq \liminf_{R \longrightarrow +\infty} \int_0^R A(s) w(s) ds = - \limsup_{R \longrightarrow + \infty} \int_{B_R} \left( \overline{\mathrm{Ric}}(N, N) + \mathrm{trace}(A^2) \right).
	\end{equation}
\end{demo}

\begin{teor}
	\label{teoGRW2}
	Let $\psi : M \longrightarrow \overline{M}$ be a complete stable maximal hypersurface in a GRW spacetime $\overline{M} = I \times_{\rho} F$ and let $\tau$, $v$ be as above.
	\begin{itemize}[noitemsep,nolistsep]
		\item [i)] If $M$ is compact, then either $\rho''(\tau)\equiv0$ on $M$ or $\rho''(\tau)$ attains both positive and negative values on $M$.
		\item [ii)] If $M$ is non-compact and for some $o\in M$ the function $v$ satisfies
		\begin{equation}
			\label{vl2}
			\lim_{r\to+\infty} \int_a^r \left( \int_{\partial B_s} v^2 \right)^{-1} ds = +\infty
		\end{equation}
		for some (hence any) $a>0$, where $B_s$ is the geodesic ball of $(M,g)$ centered at $o$ with radius $s$, then either $\rho''(\tau)\equiv0$ on $M$ or $\rho''(\tau)$ attains negative values at some points of $M$.
	\end{itemize}
\end{teor}

\begin{demo}
	Since $M$ is stable, from Lemma \ref{lemeq} there exists a positive function $u$ satisfying $\Delta u = \left( \overline{\mathrm{Ric}}(N,N) + \mathrm{trace}(A^2) \right)u$ on $M$. Set $\varphi = u/v$. By (\ref{lapvmax}), a direct computation shows that
	\begin{equation}
		\label{phieq}
		\frac{\mathrm{div}(v^2\nabla\varphi)}{v^2} = \Delta\varphi + 2g\left(\frac{\nabla v}{v},\nabla\varphi\right) = \frac{\Delta u}{v} - \frac{u\Delta v}{v^2} = -m\frac{\rho''(\tau)}{\rho(\tau)}\varphi \qquad \text{on } M.
	\end{equation}
	
	i) Suppose that $M$ is compact. We apply the divergence theorem to obtain
	$$
		0 = \int_M \mathrm{div}(v^2\nabla\varphi) = -m \int_M \frac{\rho''(\tau)}{\rho(\tau)}v^2\varphi.
	$$
	Since $v^2\varphi/\rho(\tau)>0$ on $M$, if $\rho''(\tau)\geq0$ or $\rho''(\tau)\leq0$ on $M$ then it must be $\rho''(\tau)\equiv0$.
	
	ii) Suppose that $M$ is non-compact. (\ref{vl2}) ensures that any positive function $w\in C^2(M)$ such that $\mathrm{div}(v^2\nabla w) \leq 0$ on $M$ must be constant, see Theorem 4.14 of \cite{AMR}. Since $\varphi>0$, if $\rho''(\tau) \geq 0$ on $M$ then from (\ref{phieq}) we deduce that $\varphi$ is constant and therefore $\rho''(\tau)\equiv0$.
\end{demo}

\begin{rem}
	\normalfont
	Let $\psi: M \longrightarrow \overline{M} = I \times_{\rho} F$ be a maximal hypersurface such that $\tau(M) \subseteq T_0 = \{t\in I:\rho''(t)\leq 0\}$. Since $v >0$ on $M$, by condition 3 of Lemma \ref{lemeq} and (\ref{lapvmax}) we immediately deduce that $M$ is stable. Similarly, if $\tau(M \setminus \overline{P}) \subseteq T_0$ for some relatively compact open set $P$, then $\psi$ is stable at infinity. For $T_0 = I$, stability of $\psi$ is observed in Theorem 7 of \cite{FRS}.
\end{rem}

\section{Higher order mean curvatures in Robertson-Walker spacetimes}

\label{sho}

In this section we will consider spacelike hypersurfaces in spacetimes of constant sectional curvature. Let $\psi: M \longrightarrow \overline{M}$ be such a hypersurface and suppose that the $(k+1)$-th mean curvature function $H_{k+1}$ vanishes on $M$, for some $1\leq k\leq m-2$. Then, the $k$-th Newton tensor $P_k$ corresponding to the shape operator $A$ of $\psi$ is positive definite if and only if $H_k>0$ and $\mathrm{rank}(A)>k$ on $M$, see Proposition 6.27 of \cite{BMR}. Hence, the differential operator $L_k$ defined in (\ref{Lkdef}) is elliptic if and only if these conditions are satisfied. Furthermore, since $\overline{M}$ has constant curvature, $L_k$ can be written in divergence form, that is,
\begin{equation}
	L_k u = \mathrm{div}(P_k(\nabla u)) \qquad \text{for every } u \in C^2(M).
\end{equation}
If the ambient spacetime has a Robertson-Walker spacetime structure $\overline{M} = I\times_{\rho}F$, then we can consider the functions $\eta$ and $v$ on $M$ as in the previous section. The action of $L_k$ on $\eta$ and $v$ is given by identities (\ref{Lketa}) and (\ref{Lkv}) below, also proved in Lemma 4.1 of \cite{AC} and Lemma 3.1 of \cite{CC}.

\begin{lema}
	\label{comp2}
	Let $\psi: M \longrightarrow \overline{M}$ be a spacelike hypersurface in a Robertson-Walker spacetime $\overline{M} = I \times_\rho F$ of constant curvature and let $\eta$ and $v$ be defined in (\ref{defeta}) and (\ref{defv}). For $0\leq k\leq m-1$,
	\begin{align}
		\label{Lketa}
		L_k \eta & = -c_k \rho'(\tau) H_k + c_k H_{k+1}v, \\
		\label{Lkv}
		L_k v & = \mathrm{trace}(A^2P_k) v - \binom{m}{k+1} g(T^{\top},\nabla H_{k+1}) - c_k H_{k+1}\rho'(\tau).
	\end{align}
\end{lema}

\begin{demo}
	In the proof of Lemma \ref{comp1} we have already calculated the second covariant derivatives of $\eta$ and $v$. More precisely, from (\ref{hesseta}) we can write
	\begin{align}
		\label{hesseta1}
		\nabla_X \nabla\eta & = vAX + \rho'(\tau)X, \\
		\label{hessv1}
		\nabla_X \nabla v & = (\nabla_{T^{\top}}A)X - \left(\overline{\mathrm{R}}\left(X,T^{\top}\right)N\right)^{\top} + vA^2 X + \rho'(\tau) AX
	\end{align}
	for every $X\in\mathfrak X(M)$. Recalling the definition (\ref{Lkdef}) of $L_k$, formulas (\ref{Lketa}) and (\ref{Lkv}) follow from Lemma \ref{newtop} and from the fact that $g(\overline{\mathrm{R}}\left(X,Y\right)N,Z) = \overline{\mathrm{Riem}}(Z,N,X,Y) = 0$ for every $X,Y,Z\in\mathfrak X(M)$ since $\overline{M}$ has constant curvature.
\end{demo}

The next theorem collects some observations about the $k$-stability of hypersurfaces with zero $(k+1)$-th mean curvature and positive definite $k$-th Newton tensor.

\begin{teor}
	\label{teoks}
	Let $\psi: M \longrightarrow \overline{M}$ be a complete spacelike hypersurface with zero $(k+1)$-th mean curvature, for some $1\leq k\leq m-2$, in a spacetime $\overline{M}$ of dimension $m+1$ and constant curvature $\overline{\kappa}$. Suppose that $H_k > 0$ on $M$ and that $\mathrm{rank}(A)>k$ on $M$.
	\begin{itemize}[noitemsep,nolistsep]
		\item [i)] If $\overline{\kappa}\leq0$, then $M$ is non-compact and $\psi$ is $k$-stable.
		\item [ii)] If $\overline{\kappa}>0$ and $M$ is compact and simply connected, then $\psi$ is not $k$-stable.
		\item [iii)] If $\overline{\kappa}>0$ and we assume that $M$ is non-compact and that, for some $o\in M$ and for some (hence any) $a>0$,
		\begin{equation}
			\label{ksta}
			\begin{split}
			& \lim_{r\to+\infty} \int_a^r \left( \int_{\partial B_r} H_k \right)^{-1} = +\infty, \\
			& \lim_{r \longrightarrow + \infty} \int_{B_r} \left( \overline{\kappa} \ \mathrm{trace}(P_k) - \mathrm{trace}(A^2 P_k) \right) = + \infty,
			\end{split}
		\end{equation}
		where $B_r$ is the geodesic ball of $(M,g)$ with radius $r$ centered at $o$, then $\widetilde{L}_k$ has infinite index.
	\end{itemize}
\end{teor}

\begin{demo}
	Firstly, note that \tcr{$H_{k+1}\equiv0$ and $\mathrm{rank}(A)>k$ on $M$} guarantee that the self-adjoint operator $P_k$ is positive definite (see Proposition 6.27 of \cite{BMR}). Since $A^2$ and $P_k$ are simultaneously diagonalizable, $A^2P_k$ is also self-adjoint and positive definite and therefore
	\begin{equation}
		\label{postrace}
		\mathrm{trace}(P_k) > 0, \qquad \mathrm{trace}(A^2 P_k) > 0.
	\end{equation}
	Moreover, the operator $L_k$ is elliptic and it can be put in divergence form due to the fact that $\overline{M}$ has constant sectional curvature $\overline{\kappa}$, that is, we have
	\begin{equation}
		\label{Lkdiv1}
		L_k u = \mathrm{div}(P_k(\nabla u)) \qquad \text{for every } u\in C^2(M).
	\end{equation}
	
	i) Suppose that $\overline{\kappa} \leq 0$. By (\ref{postrace}) it follows that $\left( \overline{\kappa} \ \mathrm{trace}(P_k) - \mathrm{trace}(A^2 P_k) \right) \leq 0$, so
	$$
	\int_M -\phi \widetilde{L}_k \phi = \int_M g(P_k(\nabla\phi),\nabla\phi) - \left( \overline{\kappa} \ \mathrm{trace}(P_k) - \mathrm{trace}(A^2 P_k) \right)\phi^2 \geq 0
	$$
	for every $\phi \in C^{\infty}_c(M)$, and $\psi$ is $k$-stable by definition. By Lemma \ref{lemeq}, there exists a positive function $u\in C^{\infty}(M)$ such that $L_k u = (\mathrm{trace}(A^2 P_k) - \overline{\kappa} \ \mathrm{trace}(P_k))u$ on $M$. Suppose, by contradiction, that $M$ is compact. By (\ref{Lkdiv1}) and the divergence theorem,
	$$
	0 = \int_M L_k u = \int_M (\mathrm{trace}(A^2 P_k) - \overline{\kappa} \ \mathrm{trace}(P_k))u.
	$$
	From (\ref{postrace}) and $\overline{\kappa}\leq 0$ it then follows that $u\equiv0$ on $M$, contradiction.
	
	ii) Suppose, by contradiction, that $\overline{\kappa}>0$, $M$ is compact and simply connected and $\psi$ is $k$-stable. Let $\pi : \widetilde{\overline{M}} \longrightarrow \overline{M}$ be the Lorentzian universal covering of $\overline{M}$. $\widetilde{\overline{M}}$ is isometric to de Sitter spacetime of dimension $m+1$ and curvature $\overline{\kappa}$, which in turn is isometric to the GRW spacetime $\mathbb{R}\times_{\cosh(\sqrt{\overline{\kappa}}t)}\mathbb{S}^m(\overline{\kappa})$. $M$ is simply connected, so for every $p\in M$, $\tilde p \in \pi^{-1}(p)$ there exists a unique immersion $\widetilde{\psi}:M\longrightarrow\widetilde{\overline{M}}$ such that $\widetilde{\psi}(p)=\tilde p$ and $\pi\circ\widetilde{\psi}=\psi$. The shape operator induced by $\widetilde{\psi}$ is equal to $A$, up to a change of sign, because $\pi$ is a local isometry. Then, by Lemma \ref{comp2}, $M$ supports a positive smooth functions $v$ satisfying $L_k v = \mathrm{trace}(A^2P_k)v$. Since $\psi$ is $k$-stable, by Lemma \ref{lemeq}, $M$ also supports a positive smooth function $u$ satisfying $L_k u = (\mathrm{trace}(A^2 P_k) - \overline{\kappa} \ \mathrm{trace}(P_k))u$. A direct computation shows that the positive function $\varphi = u/v$ satisfies
	$$
		\frac{\mathrm{div}(v^2 P_k(\nabla\varphi))}{v^2} = L_k \varphi + 2g\left(\frac{\nabla v}{v},P_k(\nabla\varphi)\right) = -\overline{\kappa} \ \mathrm{trace}(P_k)\varphi.
	$$
	Since $M$ is compact, we have
	$$
		0 = \int_M \mathrm{div}(v^2 P_k(\nabla\varphi)) = - \int_M \overline{\kappa} \ \mathrm{trace}(P_k) v^2\varphi.
	$$
	From (\ref{postrace}) and $\overline{\kappa}>0$ it then follows that $v^2\varphi \equiv 0$ on $M$, contradiction.
	
	iii) Suppose, by contradiction, that $\overline{\kappa}>0$, that condition (\ref{ksta}) holds for some $o\in M$ and that $\widetilde{L}_k$ has finite index. Then there exists a relatively compact open set $\Omega\subseteq M$ such that
	\begin{equation}
		\label{ksta0}
		\lambda_1^{\widetilde{L}_k}(M\setminus\overline{\Omega}) \geq 0.
	\end{equation}
	Since $(M,g)$ is complete, there exists $R>0$ such that $\overline{\Omega} \subseteq B_R$. Next we define
	\begin{equation}
		\label{ksta1}
		v_k(r) = \int_{\partial B_r} c_k H_k, \qquad A(r) = \frac{1}{v_k(r)} \int_{\partial B_r} \left( \overline{\kappa} \ \mathrm{trace}(P_k) - \mathrm{trace}(A^2 P_k) \right).
	\end{equation}
	We consider the Cauchy problem
	\begin{equation}
		\label{ksta2}
		\begin{cases}
			(v_k(r) z')' + A(r) v_k(r) z = 0 \ \text{on} \ \mathbb{R}^+ \\
			z(0^+) = Z_0 > 0, \ w(0^+) z'(0^+) = 0 \in \mathbb{R}.
		\end{cases}
	\end{equation}
	Since $v_k \in L_{\mathrm{loc}}^\infty (\mathbb{R}_0^+)$, there exists a solution $z$ of (\ref{ksta2}) with $z \in \mathrm{Lip}_{\mathrm{loc}}(\mathbb{R}_0^+)$ due to Proposition 4.2 of \cite{BMR}. Moreover, from the coarea formula and (\ref{ksta}) we obtain
	\begin{equation}
		\label{ksta3}
		\liminf_{r \longrightarrow + \infty} \int_0^r A(s) v_k(s) ds = \liminf_{r \longrightarrow + \infty} \int_{B_r} \left( \overline{\kappa} \ \mathrm{trace}(P_k) - \mathrm{trace}(A^2 P_k) \right) = + \infty.
	\end{equation}
	This condition and the fact that $v_k^{-1} \in L_{\mathrm{loc}}^\infty (\mathbb{R^+})$ and $v_k^{-1} \notin L^1(+\infty)$ enable us to use Corollary 2.9 of \cite{MMR} to obtain that any solution $z$ of (\ref{ksta2}) is oscillatory. Taking now $R \leq R_1 < R_2$ two consecutive zeros of $z$ such that $z>0$ on $(R_1, R_2)$ we define the function $\varphi(x) := z(r(x))$, where $r(x)$ is the distance from $x$ to $o$ in $(M,g)$, and compute
	\begin{equation}
		\label{kkk4}
		\begin{split}
			\int_{B_{R_2} \setminus \overline{B_{R_1}}} - \varphi \widetilde{L}_k \varphi & = \int_{B_{R_2} \setminus \overline{B_{R_1}}} g \left( P_k(\nabla \varphi), \nabla \varphi \right) - \left( \overline{\kappa} \ \mathrm{trace}(P_k) - \mathrm{trace}(A^2 P_k) \right) \varphi^2 \\
			& \leq \int_{B_{R_2} \setminus \overline{B_{R_1}}} \mathrm{trace}(P_k) |\nabla \varphi|^2 - \left( \overline{\kappa} \ \mathrm{trace}(P_k) - \mathrm{trace}(A^2 P_k)
			\right) \varphi^2.
		\end{split}
	\end{equation}
	With the aid of the coarea formula, integrating by parts and using (\ref{ksta2}) we obtain 
	\begin{equation}
		\label{ksta5}
		\int_{B_{R_2} \setminus \overline{B_{R_1}}} - \varphi \widetilde{L}_k \varphi \leq - \int_{R_1}^{R_2} \left[ (v_k z')'(s) + A(s) v_k(s) z(s) \right] z(s) ds = 0.
	\end{equation}
	Therefore, from (\ref{ksta5}) and Proposition \ref{propo2} we deduce
	\begin{equation}
		\label{ksta6}
		\lambda_1^{\widetilde{L}_k} (M \setminus \overline{\Omega}) < \lambda_1^{\widetilde{L}_k} (B_{R_2} \setminus \overline{B_{R_1}}) \leq \frac{\int_{B_{R_2} \setminus \overline{B_{R_1}}} - \varphi \widetilde{L}_k \varphi}{\int_{B_{R_2}\setminus \overline{B_{R_1}}}\varphi^2} \leq 0,
	\end{equation}
	contradicting (\ref{ksta0}).
\end{demo}

\begin{teor}
	\label{kmax2}
	Let $\overline{M} = I \times_\rho F$ be a Robertson-Walker spacetime of constant sectional curvature and let $\psi: M \longrightarrow \overline{M}$ be a complete non-compact spacelike hypersurface with zero $(k+1)$-th mean curvature. Suppose that $H_k > 0$, $\sup_M H_k < +\infty$, $\mathrm{rank}(A) > k$ on $M$ and that, for some $o\in M$ and for some (hence any) $a\in\mathbb{R}$, one of the following conditions is satisfied:
	\begin{alignat*}{3}
		(i) & \qquad \lim_{r\to+\infty} \int_a^r \left( \int_{\partial B_r} H_k \right)^{-1} dt = +\infty && \qquad \text{and} \qquad \lim_{r\to+\infty} \int_{B_r} H_k = +\infty, \qquad \text{or} \\
		(ii) & \qquad \lim_{r\to+\infty} \int_a^r \left( \int_{\partial B_r} H_k \right)^{-1} dt < +\infty && \qquad \text{and} \qquad \lim_{r\to+\infty} \frac{1}{r}\log\left( \int_{\partial B_r} H_k \right) = 0,
	\end{alignat*}
	where $B_r$ is the geodesic ball of $(M,g)$ with radius $r$ centered at $o$. If $\psi(M)$ is contained in a slab $[a,b] \times F$, then $\inf_M |H_k\rho'(\tau)| = 0$. In particular, if $\inf_M H_k > 0$, then $\inf_M |\rho'(\tau)| = 0$ and therefore there exists $t_0$ in the closure of $\tau(M)\subseteq[a,b]$ such that $\rho'(t_0) = 0$.
\end{teor}

\begin{demo}
	First observe that $H_{k+1}\equiv0$, $H_k>0$ and $\mathrm{rank}(A)>k$ guarantee that $P_k$ is positive definite and therefore $L_k$ is elliptic (see Proposition 6.27 of \cite{BMR}). Set $w(r) = \int_{\partial B_r} \mathrm{trace}(P_k)$ for every $r>0$ and $\Lambda = \sup_M H_k$.
	
	Let $R>0$ be given and suppose that condition (i) is satisfied. Then
	$$
	\lim_{r\to+\infty}\int_a^r\frac{dt}{w(t)} = \lim_{r\to+\infty} \int_a^r w(t)dt = +\infty,
	$$
	since $\mathrm{trace}(P_k) = c_k H_k$. Let $\varepsilon>0$ be given. By Theorem 6.6 of \cite{BMR}, the solution $z \in \mathrm{Lip}_{\mathrm{loc}}([R,+\infty))$ of the Cauchy problem
	\begin{equation}
		\label{CPkmax}
		\begin{cases}
			(w(r)z')' + \varepsilon w(r)z = 0 & \text{on } [R,+\infty) \\
			z(R) = 1, \quad (wz')(R^+) = 0
		\end{cases}
	\end{equation}
	is oscillatory. Let $R<R_1<R_2$ be two consecutive zeros of $z$ such that $z>0$ on $(R_1,R_2)$ and let $\varphi(x)=z(r(x))$, where $r(x)$ is the distance from $x$ to $o$ in $(M,g)$. By the coarea formula, the inequality $\mathrm{Vol}(\partial B_r) \geq w(r)/\Lambda$ and the fact that $z$ solves (\ref{CPkmax}), we have
	\begin{align*}
		\lambda_1^{L_k}(B_{R_2}\setminus\overline{B_{R_1}}) & \leq \frac{\int_{B_{R_2}\setminus\overline{B_{R_1}}}g(P_k(\nabla\varphi),\nabla\varphi)}{\int_{B_{R_2}\setminus\overline{B_{R_1}}}\varphi^2} = \frac{\int_{R_1}^{R_2}w(r)z'(r)^2}{\int_{R_1}^{R_2}\mathrm{Vol}(\partial B_r)z(r)^2} \\
		& \leq \frac{\int_{R_1}^{R_2}\varepsilon w(r)z(r)^2}{\int_{R_1}^{R_2}\frac{w(r)}{\Lambda}z(r)^2} = \varepsilon\Lambda,
	\end{align*}
	and by Proposition \ref{propo2} we get
	\begin{equation}
		\label{klambdaepsilon}
		0 \leq \lambda_1^{L_k}(M\setminus\overline{B_R}) < \varepsilon\Lambda.
	\end{equation}
	
	Suppose now that (ii) is satisfied. Then $\lim_{r\to+\infty}\int_a^r\frac{dt}{w(t)}=0$ and for every $\varepsilon>0$ there exists $C_{\varepsilon}>0$ such that $w(r)\leq C_{\varepsilon}e^{\varepsilon r}$ for every $r>0$. Let $0<\varepsilon<4$ be given. Then, for every $r>0$,
	\begin{align*}
		\varepsilon & > \frac{\varepsilon^2}{4} = \left( \lim_{s\to+\infty} \left(-\frac{1}{2(s-r)}\log\int_s^{+\infty}\frac{dt}{C_{\varepsilon}e^{\varepsilon t}}\right) \right)^2 \\
		& \geq \left( \inf_{s>r} \left(-\frac{1}{2(s-r)}\log\int_s^{+\infty}\frac{dt}{C_{\varepsilon}e^{\varepsilon t}}\right) \right)^2
	\end{align*}
	and by Proposition 6.9 of \cite{BMR} this condition is sufficient to deduce that for every $R>0$ the solution $z$ of the Cauchy problem (\ref{CPkmax}) is oscillatory, so inequality (\ref{klambdaepsilon}) follows again.
	
	Letting $\varepsilon\to 0^+$ and $R\to 0^+$, we deduce that $\lambda_1^{L_k}(M)=0$. We have $\psi(M)\subseteq[a,b]\times F$, with $a,b\in I$. Choose $t_0=a$ and let $\overline{\eta}$, $\eta$ be as in (\ref{defeta}). Then $0=\overline{\eta}(a)\leq\eta\leq\overline{\eta}(b)<+\infty$ on $M$. Fix $\varepsilon>0$. By Proposition \ref{barta} applied to the positive functions $\eta+\varepsilon$, $\overline{\eta}(b)+\varepsilon-\eta$ and by (\ref{Lketa}) we have
	$$
		0 \geq \inf_M \left( - \frac{L_k\eta}{\eta+\varepsilon} \right) = \inf_M \frac{c_k H_k \rho'(\tau)}{\eta + \varepsilon}, \quad 0 \geq \inf_M \left( \frac{L_k\eta}{\overline{\eta}(b)+\varepsilon-\eta} \right) = \inf_M \frac{- c_k H_k \rho'(\tau)}{\overline{\eta}(b) + \varepsilon - \eta}.
	$$
	Since $\eta+\varepsilon>\varepsilon$, $\overline{\eta}(b)+\varepsilon-\eta>\varepsilon$ and $c_k>0$, we deduce $\inf_M |H_k \rho'(\tau)| = 0$.
\end{demo}

\begin{rem}
	Under the hypothesis $C^{-1} < H_k < C$ we have $c_k C^{-1}\mathrm{Vol}(\partial B_r) \leq w(r) \leq c_k C\mathrm{Vol}(\partial B_r)$ for every $r>0$ and therefore conditions (i) and (ii) in the statement of Theorem \ref{intro:kmax2} are equivalent to the corresponding conditions in the statement of Theorem \ref{kmax2}.
\end{rem}

\section*{Acknowledgements} 

The second author is supported by Spanish MINECO and ERDF project MTM2016-78807-C2-1-P.


\begin{thebibliography}{99}

\bibitem{AA} A.L. Albujer, L.J. Al\'ias, Calabi-Bernstein results for maximal surfaces in Lorentzian product spaces, \emph{J. Geom. Phys.}, \textbf{59} (2009), 620--631.

\bibitem{ABC} L.J. Al\'ias, A. Brasil and A.G. Colares, Integral formulae for spacelike hypersurfaces in conformally stationary spacetimes and applications, \emph{Proc. Edinb. Math. Soc. (2)}, \textbf{46} (2003), 465--488.

\bibitem{AC} L.J. Al\'ias and A.G. Colares, Uniqueness of spacelike hypersurfaces with constant higher order mean curvature in generalized Robertson-Walker spacetimes, \emph{Math. Proc. Cambridge}, \textbf{143} (2007), 703--729.

\bibitem{AMR} L.J. Al\'ias, P. Mastrolia and M. Rigoli, \emph{Maximum principles and geometric applications}, Springer (2016).

\bibitem{ARS2} L.J. Al\'ias, A. Romero and M. S\'anchez, Spacelike hypersurfaces of constant mean curvature and Calabi-Bernstein type problems, \emph{T\^ohoku Math. J.}, \textbf{49} (1997), 337--345.

\bibitem{ARS} L.J. Al\'ias, A. Romero and M. S\'anchez, Uniqueness of complete spacelike hypersurfaces of constant mean curvature in Generalized Robertson-Walker spacetimes, \emph{Gen. Relativ. Gravit.}, \textbf{27} (1995), 71--84.

\bibitem{ARSc} L.J. Al\'ias, M. Rigoli and S. Scoleri, Weak maximum principles and geometric estimates for spacelike hypersurfaces in generalized Robertson-Walker spacetimes, \emph{Nonlinear Anal.} \textbf{129} (2015), 119--142.

\bibitem{AMM} J.M. Arms, J.E. Marsden and V. Moncrief, The structure of the space of solutions of Einstein's equations. II. Several Killing fields and the Einstein-Yang-Mills equations, \emph{Ann. Phys.}, \textbf{144} (1982), 81--106. 

\bibitem{Bar} R. Bartnik, Existence of maximal surfaces in asymptotically flat spacetimes, \emph{Commun. Math. Phys.}, \textbf{94} (1984), 155--175.

\bibitem{BMR1} B. Bianchini, L. Mari and M. Rigoli, Spectral radius, index estimates for Schr\"odinger operators and geometric applications, \emph{J. Funct. Anal.}, \textbf{256} (2009), 1769--1820.

\bibitem{BMR} B. Bianchini, L. Mari and M. Rigoli, On some aspects of oscillation theory and geometry, \emph{Mem. Am. Math. Soc.}, \textbf{225} (2013), vi+195.

\bibitem{brasil} A. Brasil and A.G. Colares, On constant mean curvature spacelike hypersurfaces in Lorentz manifolds, \emph{Mat. Contemp.}, \textbf{17} (1999) 99--136.

\bibitem{BC} A. Brasil and A.G. Colares, Stability of spacelike hypersurfaces with constant $r$-mean curvature in de Sitter space, \emph{Proceedings of the XII Fall Workshop on Geometry and Physics}, \emph{Publ. R. Soc. Mat. Esp.}, \textbf{7} (2004), 139--145. 

\bibitem{BF} D. Brill and F. Flaherty, Isolated maximal surfaces in spacetime, \emph{Commun. Math. Phys.} \textbf{50} (1984), 157--165.

\bibitem{Ca} E. Calabi, Examples of Bernstein problems for some nonlinear equations, \emph{P. Symp. Pure Math.}, \textbf{15} (1970), 223--230.

\bibitem{CC} F. Camargo, A. Caminha, M. da Silva and H. de Lima, On the r-stability of spacelike hypersurfaces, \emph{J. Geom. Phys.}, \textbf{60} (2010), 1402-1410.

\bibitem{CY} S.Y. Cheng and S.T. Yau, Maximal spacelike hypersurfaces in the Lorentz-Minkowski spaces, \emph{Ann. of Math.},  \textbf{104} (1976), 407--419.

\bibitem{Ch} Y. Choquet-Bruhat, Quelques propri\'et\'es des sousvari\'et\'es maximales d'une vari\'et\'e lorentzienne, \emph{Cr. Acad. Sci. A Math. (Paris)  Serie A}, \textbf{281} (1975), 577--580.

\bibitem{CG} Y. Choquet-Bruhat and R. Geroch, Global Aspects of the Cauchy Problem in General Relativity, \emph{Commun. Math. Phys.}, \textbf{14} (1969), 329--335.  

\bibitem{F} T. Frankel, Applications of Duschek's formula to cosmology and minimal surfaces, \emph{B. Am. Math. Soc.}, \textbf{81} (1975), 579--583.

\bibitem{FRS} D. de la Fuente, R.M. Rubio and J.J. Salamanca, Stability of maximal hypersurfaces in spacetimes: new general conditions and applications to relevant spacetimes, \emph{Gen. Relativ. Gravit.}, \textbf{49} (2017), 129--143.

\bibitem{FC} D. Fischer-Colbrie, On complete minimal surfaces with finite Morse index in three-manifolds, \emph{Invent. Math.}, \textbf{82} (1985), 121--132.

\bibitem{FS} D. Fischer-Colbrie and R. Schoen, The structure of complete stable minimal surfaces in 3-manifolds of non-negative scalar curvature, \emph{Commun. Pur. Appl. Math.}, \textbf{33} (1980), 199--211.

\bibitem{GT} D. Gilbarg, N. Trudinger, \emph{Elliptic partial differential equations of second order}, Springer (2001).

\bibitem{HL} J. Hounie, M.L. Leite, The maximum principle for hypersurfaces with vanishing curvature functions, \emph{J. Differential Geom.}, \textbf{41} (1995), 247--258.

\bibitem{I} T. Ishihara, Maximal spacelike submanifolds of a pseudoriemannian space of constant curvature, \emph{Michigan Math. J.}, \textbf{35} (1988), 345--352.

\bibitem{JKG} J.L. Jaramillo, J.A.V. Kroon and E. Gourgoulhon, From geometry to numerics: interdisciplinary aspects in mathematical and numerical relativity, \emph{Classical Quant. Grav.}, \textbf{25} (2008), 093001.

\bibitem{Li} A. Lichnerowicz, L'integration des \'equations de la gravitation relativiste et le probl\`eme des n corps, \emph{J. Math. Pure Appl.}, \textbf{23} (1944), 37--63.

\bibitem{MMR} L. Mari, P. Mastrolia and M. Rigoli, A note on Killing vector fields and CMC hypersurfaces, \emph{J. Math. Anal. Appl.}, \textbf{431} (2015), 919--934.

\bibitem{M-T} J.E. Marsden and F.J. Tipler, Maximal hypersurfaces and foliations of constant mean curvature in General Relativity, \emph{Phys. Rep.}, \textbf{66} (1980), 109--139.

\bibitem{MRS} P. Mastrolia, M. Rigoli and A.G. Setti, \emph{Yamabe-type equations on complete, noncompact manifolds}, Springer, (2012).

\bibitem{N} S. Nishikawa, On maximal spacelike hypersurfaces in a Lorentzian manifold, \emph{Nagoya Math. J.}, \textbf{95} (1984), 117--124.

\bibitem{O'N} B. O'Neill, \emph{Semi-Riemannian Geometry with applications to Relativity}, Academic Press, (1983).

\bibitem{Os} R. Osserman, \emph{A survey of minimal surfaces}, Dover Publications, (1986). 

\bibitem{PRR1} J.A.S. Pelegr\'in, A. Romero and R.M. Rubio, On maximal hypersurfaces in Lorentz manifolds admitting a parallel lightlike vector field, \emph{Classical Quant. Grav.}, \textbf{33} (2016), 055003(1--8).

\bibitem{PRR3} J.A.S. Pelegr\'in, A. Romero and R.M. Rubio, On uniqueness of the foliation by comoving observers restspaces of a Generalized Robertson-Walker spacetime, \emph{Gen. Relativ. Gravit.}, \textbf{49} (2017), Art. 16, 14pp.

\bibitem{PRR2} J.A.S. Pelegr\'in, A. Romero and R.M. Rubio, Uniqueness of complete maximal hypersurfaces in spatially open $(n+1)$-dimensional Robertson-Walker spacetimes with flat fiber, \emph{Gen. Relativ. Gravit.}, \textbf{48} (2016), 1--14.

\bibitem{PRS} S. Pigola, M. Rigoli and A.G. Setti, Vanishing theorems on Riemannian manifolds, and geometric applications, \emph{J. Funct. Anal.}, \textbf{229} (2005), 424--461.

\bibitem{PRS2} S. Pigola, M. Rigoli and A.G. Setti, \emph{Vanishing and finiteness results in geometric analysis}, Birkh\"auser (2008).

\bibitem{RS} M. Rigoli and A.G. Setti, Liouville type theorems for $\phi$-subharmonic functions, \emph{Rev. Mat. Iberoam}, \textbf{17} (2001), 471--520.

\bibitem{SY} R. Schoen and S.T. Yau, On the proof of the positive mass conjecture in General Relativity, \emph{Comm. Math. Phys.} \textbf{65} (1979), 45--76.

\bibitem{JS} J. Simons, Minimal varieties in riemannian manifolds, \emph{Ann. Math.}, \textbf{88} (1968), 62--105.

\bibitem{T} P. Tolksdorf, Regularity of a more general class of quasilinear elliptic equations, \emph{J. Differential Equations}, \textbf{51} (1984), 126--150.

\bibitem{W} J.A. Wolf, \emph{Spaces of constant curvature}, McGraw-Hill (1967).

\end{thebibliography}
\end{document}